\documentclass[a4paper,11pt]{amsart}

\usepackage{amsmath,amssymb,amsthm,graphicx,pstricks,comment}
\usepackage[all]{xy}
\usepackage{lscape}

\title{Preprojective algebras and c-sortable words}
\author{Claire Amiot}
\address{Institut de Recherche de Math\'ematique Avanc\'ee, 7 rue Ren\'e Descartes, 67084 Strasbourg Cedex, France}
\email{amiot@math.unistra.fr}
\thanks{All authors were supported by the Storforsk-grant 167130 from the Norwegian Research Council}
\author{Osamu Iyama}
\address{Graduate School of Mathematics, Nagoya University}
\email{iyama@math.nagoya-u.ac.jp}
\thanks{The second author was supported by JSPS Grant-in-Aid for Scientific Research 21740010 and 21340003}

\author{Idun Reiten}
\address{Insitutt for matematiske fag,
Norges Teknisk-Naturvitenskapelige Universitet,
N-7491 Trondheim, Norway}
\email{idunr@math.ntnu.no}

\author{Gordana Todorov}
\address{Departement of Mathematics, Northeastern University, Boston, MA 02115, USA}
\email{g.todorov@neu.edu}
\thanks{The fourth author was also supported by the NSA-grant MSPF-08G-228}

\newcommand{\Hom}{{\sf Hom }}\newcommand{\Rad}{{\sf Rad }}
\newcommand{\End}{{\sf End }}
\newcommand{\Ext}{{\sf Ext }}
\newcommand{\Tor}{{\sf Tor}}
\newcommand{\Ker}{{\sf Ker }}
\newcommand{\Coker}{{\sf Coker }}
\newcommand{\Imm}{{\sf Im }}
\renewcommand{\mod}{{\sf mod \hspace{.02in}  }}\newcommand{\fl}{{\sf f.l. \hspace{.02in}  }}
\newcommand{\Mod}{{\sf Mod \hspace{.02in} }}

\newcommand{\proj}{{\sf proj \hspace{.02in} }}

\newcommand{\Sub}{{\sf Sub \hspace{.02in} }}
\newcommand{\Fac}{{\sf Fac \hspace{.02in} }}
\newcommand{\add}{{\sf add \hspace{.02in} }}

\newcommand{\ten}{\otimes}
\newcommand{\lten}{\overset{\mathbf{L}}{\ten}}

\newcommand{\Cc}{\mathcal{C}}
\newcommand{\Dd}{\mathcal{D}}

\newcommand{\Ff}{\mathcal{F}}

\newcommand{\Kk}{\mathcal{K}}

\newcommand{\ww}{\mathbf{w}}

\newcommand{\bsm}{\begin{smallmatrix}}
\newcommand{\esm}{\end{smallmatrix}}

\newtheorem{thma}{Theorem}[section]
\newtheorem{lema}[thma]{Lemma}
\newtheorem{cora}[thma]{Corollary}
\newtheorem{prop}[thma]{Proposition}
\theoremstyle{remark}
\newtheorem{rema}[thma]{Remark}
\newtheorem{exa}[thma]{Example}
\theoremstyle{definition}
\newtheorem{dfa}[thma]{Definition}
\newtheorem{pb}[thma]{Problem}

\begin{document}

\begin{abstract}
 Let $Q$ be an acyclic quiver and $\Lambda$ be the complete preprojective algebra of $Q$ over an algebraically closed field $k$. To any element $w$ in the Coxeter group of $Q$, Buan, Iyama, Reiten and Scott have introduced and studied in \cite{Bua2} a finite dimensional algebra $\Lambda_w=\Lambda/I_w$. In this paper we look at filtrations of $\Lambda_w$ associated to any reduced expression $\ww$ of $w$. We are especially interested in the case where the word $\ww$ is $c$-sortable, where $c$ is a Coxeter element. In this situation, the consecutive quotients of this filtration can be related to tilting $kQ$-modules with finite torsionfree class.
\end{abstract}

\maketitle
\tableofcontents
\section*{Introduction}
Attempts to categorify the cluster algebras of Fomin and Zelevinsky \cite{FZ1} have led to the investigation of categories with the 2-Calabi-Yau property (2-CY for short) and their cluster-tilting objects. Main early classes of examples were the cluster categories associated with finite dimensional path algebras \cite{Bua} and the preprojective algebras of Dynkin type \cite{Gei3}. This paper is centered around the more general class of stably 2-CY and triangulated 2-CY categories associated with elements in Coxeter groups \cite{Bua2} (the adaptable case was done independently in \cite{Gei2}), and their relationship to the generalized cluster categories from \cite{Ami3} (see Section 4 for the definition).

Let $Q$ be a finite connected quiver with vertices $1,\ldots,n$, and $\Lambda$ the complete preprojective algebra of the quiver $Q$ over a field $k$. Denote by $s_1,\ldots,s_n$ the distinguished generators in the corresponding Coxeter group $W_Q$. To an element $w$ in $W_Q$, there is associated a stably 2-CY category $\Sub\Lambda_w$ and a triangulated 2-CY category $\underline{\Sub}\Lambda_w$. The definitions are based on first associating an ideal $I_i$ in $\Lambda$ to each $s_i$, hence to any reduced word by taking products. This way we also get a finite dimensional algebra $\Lambda_w:=\Lambda/I_w$. Objects of the category $\Sub \Lambda_w$ are submodules of finite dimensional free $\Lambda_w$-modules. The cluster category is then equivalent to $\underline{\Sub}\Lambda_w$ with $w=c^2$, where $c$ is a Coxeter element such that $c^2$ is a reduced expression \cite{Bua2,Gei2}. When $\Lambda$ is a preprojective algebra of Dynkin type, then the category $\mod \Lambda$ as 
 investigated in \cite{Gei3} is also obtained as $\Sub\Lambda_w$ where $w$ is the longest element \cite[III 3.5]{Bua2}.

Using the construction of ideals we get for each reduced expression $\mathbf{w}=s_{u_1}s_{u_2}\ldots s_{u_l}$ a chain of ideals $$\Lambda\supset I_{u_1}\supset I_{u_2}I_{u_1}\supset \ldots \supset I_w,$$ which gives rise to an interesting set of $\Lambda$-modules:
$$L^1_\ww:=\frac{\Lambda}{I_{u_1}},\ L^2_\ww:=\frac{I_{u_1}}{I_{u_2}I_{u_1}}, \ldots, L^l_\ww:=\frac{I_{u_{l-1}}\ldots I_{u_1}}{I_w}$$
which all turn out to be indecomposable and to lie in $\Sub\Lambda_w$. 

The investigation of this set of modules, which we call \emph{layers}, from different points of view, including connections with tilting theory, is one of the main themes of this paper, especially for a class of words called $c$-sortable.

\medskip
The modules $L^1_\ww,\ldots, L^l_\ww$ provide a natural filtration for the cluster-tilting object $M_\ww$ associated with the reduced expression $\ww=s_{u_1}\ldots s_{u_l}$ (see Section 1). These modules can be used to show that the endomorphism algebras $\End_\Lambda (M_\ww)$ are quasi-hereditary \cite{IR2}. Here we show that these modules are rigid (Theorem \ref{2spherical}), that is $\Ext^1_\Lambda(L_\ww^j,L^j_\ww)=0$ and that their dimension vectors are real roots (Theorem \ref{realroot}), so that there are unique associated indecomposable $kQ$-modules $(L^j_\ww)_Q$ (which are not necessarily rigid).

\medskip
The situation is especially nice when all layers are indecomposable $kQ$-modules, so that $L^j_\ww=(L^j_\ww)_Q$. This is the case for $c$-sortable words.  An element $w$ of $W_Q$ is $c$-sortable when there exists a reduced expression of $w$ of the form $\ww=c^{(0)}c^{(1)}\ldots c^{(m)}$ with $$supp(c^{(m)})\subseteq\ldots \subseteq supp(c^{(1)})\subseteq supp(c^{(0)})\subseteq supp(c),$$ where $c$ is a Coxeter element, that is, a word containing each generator $s_i$ exactly once, and in an order admissible with respect to the orientation of $Q$. 

Starting with the tilting $kQ$-module $kQ$ (when $c^{(0)}=c$), there is a natural way of performing exchanges of complements of almost complete tilting modules, determined by the given reduced expression. We denote the final tilting module by $T_\ww$, and the indecomposable $kQ$-modules used in the sequence of constructions by $T^j_\ww$ for $j=1,\ldots ,l$. We show that $L^j_\ww$ is a $kQ$-module in this case and that $L^j_\ww\simeq T^j_\ww$ for all $j$ (Theorem \ref{LequalT}) and we also show that the indecomposable modules in the torsionfree class $\Sub (T_\ww)$ are exactly the $T^j_\ww$ (Theorem \ref{Twtilting}). In particular this gives a one-one correspondence between $c$-sortable words and torsionfree classes, as first shown in \cite{Tho} using different methods (see also \cite{IT}).

There is another sequence $U^1_\ww,\ldots U^l_\ww$ of indecomposable $kQ$-modules, defined using restricted reflection functors, which coincide with the above sequences. This is both interesting in itself, and provides a method for proving $L^j_\ww\simeq T^j_\ww$ for $j=1,\ldots ,l$. 

In another paper \cite{AIRT2}, we give a description of the layers from a functorial point of view. When the $c$-sortable word is $c^m$, and $c=s_1\ldots s_n$, then the successive layers are given by $$P_1,\ldots, P_n, \tau^{-}P_1, \ldots, \tau^{-}P_n,\tau^{-2}P_1, \ldots, \tau^{-m}P_n$$ for the indecomposable projective $kQ$-modules $P_i$, where $\tau$ denotes the AR-translation. In the general case we will give a description of the layers using specific factor modules of the above modules. 

The generalized cluster categories $\Cc_A$ for algebras $A$ of global dimension at most two were introduced in \cite{Ami3}. It was shown that for a special class of words $w$, properly contained in the dual of the $c$-sortable words, the 2-CY category $\underline{\Sub}\Lambda_w$ is triangle equivalent to some $\Cc_A$. We point out  that the procedure for choosing $A$ works more generally for any dual of a $c$-sortable word (Theorem \ref{equivalence}).

\medskip
The paper is organized as follows. We start with some background material on 2-CY categories associated with reduced words, on complements of almost complete tilting modules and on reflection functors. In Section 2 we show that for any reduced word $w$, the associated layers are indecomposable rigid modules, which also are positive real roots. Hence there are unique associated indecomposable $kQ$-modules. In Section 3 we show that our three series of indecomposable modules $\{L^j_\ww\}$, $\{T^j_\ww\}$ and $\{U^j_\ww\}$ coincide in the $c$-sortable case.  Section  4 is devoted to examples and questions beyond the $c$-sortable case.

Some of this work was presented at the conferences `Homological and geometric methods in representation theory' in Trondheim in August 2009, `Interplay between representation theory and geometry' in Beijing in May 2010, and in seminars in Bonn and Torun in 2010.

\subsection*{Notation}
Throughout $k$ is an algebraically closed field. The tensor product $-\ten-$, when not specified, will be over the field $k$. For a $k$-algebra $A$, we denote by $\mod A$ the category of finitely presented right $A$-modules, and by $\fl A$ the category of finite length right $A$-modules. For a quiver $Q$ we denote by $Q_0$ the set of vertices and by $Q_1$ the set of arrows, and for $a\in Q_1$ we denote by $s(a)$ its source and by $t(a)$ its target.

\subsection*{Acknowledgements}
This work was done when the first author was a postdoc at NTNU, Trondheim. She would like to thank the Research Council of Norway for financial support.
Part of this work was done while the second author visited NTNU during March and August 2009. He would like to thank the people in Trondheim for their hospitality.

The authors would also like to thank the referee for reading the paper thoroughly, and for several useful suggestions. 
\section{Background}

\subsection{2-Calabi-Yau categories associated with reduced words}
Let $Q$ be a finite connected quiver without oriented cycles and with vertices $Q_0=\{1,\ldots, n\}$. 
For $i,j\in Q_0$ we denote by $m_{ij}$ the positive integer
$$m_{ij}:=\sharp\{a\in Q_1|\ s(a)=i, t(a)=j\}+\sharp\{a\in Q_1|\ s(a)=j, t(a)=i\}.$$
The \emph{Coxeter group} associated to $Q$ is defined by the generators $s_1,\ldots , s_n$ and relations
\begin{itemize}
 \item $s_i^2=1$,
\item $s_is_j=s_js_i$ if $m_{ij}=0$,
\item $s_is_js_i=s_js_is_j$ if $m_{ij}=1$.
\end{itemize}
In this paper $\ww$ will denote a word (\emph{i.e.} an expression in the free monoid generated by $s_i, i\in Q_0$), and $w$ will be its equivalence class in the Coxeter group $W_Q$.

An expression $\ww=s_{u_1}\ldots s_{u_l}$ is \emph{reduced} if $l$ is smallest possible. An element $c=s_{u_1}\ldots s_{u_l}$ is called a \emph{Coxeter element} if $l=n$ and $\{u_1,\ldots, u_l\}=\{1,\ldots ,n\}$. We say that a Coxeter element $c=s_{u_1}\ldots s_{u_n}$ is \emph{admissible} with respect to the orientation of $Q$ if $i<j$ when there is an arrow $u_i\rightarrow u_j$.

\medskip 
The preprojective algebra associated to $Q$ is the algebra
$$k\overline{Q}/\langle \sum_{a\in Q_1} (aa^*-a^*a)\rangle$$
where $\overline{Q}$ is the double quiver of $Q$, which is obtained from $Q$ by adding for each arrow 
$a:i\rightarrow j$ in $Q_1$ an arrow $a^*:i\leftarrow j$ pointing in the opposite direction. 
We denote by $\Lambda$ the completion of the preprojective algebra associated to $Q$ and by $\fl\Lambda$ 
the category of right $\Lambda$-modules of finite length. 

The algebra $\Lambda$ is finite-dimensional selfinjective if $Q$ is a Dynkin quiver. Then the stable category $\underline{\mod}\Lambda$ satisfies the \emph{2-Calabi-Yau property} (2-CY for short), that is, there is a functorial isomorphism  $$D\underline{\Hom}_{\Lambda}(X,Y)\simeq \underline{\Hom}_{\Lambda}(Y,X[2]),$$ where $D:=\Hom_k(-,k)$ and $[1]:=\Omega^{-1}$ is the suspension functor \cite{AR96,CB00} (see \cite{GLS07,Rei06} for a complete proof). 

When $Q$ is not Dynkin, then $\Lambda$ is infinite dimensional and of global dimension $2$. In this case the triangulated category $\Dd^b(\fl\Lambda)$ is $2$-CY \cite{CB00,GLS07,Boc08} (see \cite{Y}  for a complete proof).

We now recall some work from \cite{IR, Bua2}.
For each $i=1,\ldots, n$ we have an ideal $I_i:=\Lambda(1-e_i)\Lambda$ in $\Lambda$, where $e_i$ is the idempotent of $\Lambda$ associated with the vertex $i$. We write $I_\ww:= I_{u_l}\ldots I_{u_2}I_{u_1}$ when $\ww=s_{u_1}s_{u_2}\ldots s_{u_l}$ is a expression of $w\in W_Q$. 
We denote by $S_i:=\Lambda/I_i$ the simple bimodule corresponding to the vertex $i$.

We collect the following information which is useful for Section 2:

\begin{prop}\cite{Bua2}\label{basicprop} Let $\Lambda$ be a complete preprojective algebra.

\begin{itemize}
\item[(a)] If $\ww=s_{u_1}\ldots s_{u_l}$ and $\ww'=s_{v_1}\ldots s_{v_l}$ are two reduced expressions of the same element in the Coxeter group, then $I_\ww=I_{\ww'}$. 

\item[(b)] If $\ww=\ww's_i$ with $\ww'$ reduced, then $I_\ww\subseteq I_{\ww'}$. Moreover $\ww$ is reduced if and only if $I_\ww\varsubsetneq I_{\ww'}$.  And for $j\neq i$ we have $e_j I_{\ww}=e_jI_{\ww'}$.

\end{itemize}

If $\Lambda$ is not of Dynkin type we have moreover:
 \begin{itemize}

\item[(c)] Any finite product $I$ of the ideals $I_j$ is a tilting module of projective dimension at most one, and $\End_\Lambda(I)\simeq \Lambda$.

\item[(d)] If $S$ is a simple $\Lambda$-module and $I$ is a tilting module of projective dimension at most one, then $S\ten_\Lambda I=0$ or $\Tor_1^\Lambda(S,I)=0$.

\item[(e)] If $\Tor_1^\Lambda(S_i,I)=0$, then $I_i\lten_\Lambda I=I_i\ten_\Lambda I=I_iI$ for a tilting module $I$ of projective dimension at most one.

\end{itemize}
\end{prop}

By $(a)$ the ideal $I_\ww$ does not depend on the choice of the reduced expression $\ww$ of $w$. Therefore we write $I_w$ for the ideal $I_\ww$ and $\Lambda_w:=\Lambda/I_w$ when $\ww$ is an expression of $w$. This is a finite dimensional algebra. We denote by $\Sub\Lambda_w$ the category of submodules of finite dimensional free $\Lambda_w$-modules. This is a \emph{Frobenius category}, that is, an exact category with enough projectives and injectives, and the projectives and injectives coincide. Its stable category $\underline{\Sub}\Lambda_w$ is a triangulated category which satisfies the 2-Calabi-Yau property \cite{Bua2}. The category $\Sub\Lambda_w$ is then said to be \emph{stably 2-Calabi-Yau}.

Recall that a \emph{cluster-tilting object} in a Frobenius stably 2-CY category $\Cc$ with finite dimensional morphisms spaces is an object $T\in \Cc$ such that \begin{itemize}
\item $\Ext^1_\Cc(T,T)=0$
\item $\Ext^1_\Cc(T,X)=0$ implies that $X\in \add T$.
\end{itemize}

For any reduced word $\ww=s_{u_1}\ldots s_{u_l}$, we write ${\displaystyle M^j_\ww:=e_{u_j}\frac{\Lambda}{I_{u_j}\ldots I_{u_1}}.}$ 

\begin{thma}\cite[Thm III.2.8]{Bua2}\label{birsc}
For any reduced expression $\ww=s_{u_1}\ldots s_{u_l}$ of $w\in W_Q$, the object $M_\ww:=\bigoplus_{j=1}^l M^j_\ww$ 
is a cluster-tilting object in the stably 2-CY category $\Sub\Lambda_w$. 
\end{thma}

For any reduced word $\ww=s_{u_1}\ldots s_{u_l}$, we have the chain of ideals $$\Lambda\supset I_{u_1}\supset I_{u_2}I_{u_1}\supset \ldots \supset I_w,$$ which is strict by Proposition \ref{basicprop} (b). For $j=1,\ldots,l$ we define the \emph{layer}  $$L^j_\ww:=\frac{I_{u_{j-1}}\ldots I_{u_1}}{I_{u_j}\ldots I_{u_1}}.$$

Using Proposition \ref{basicprop} (b) it is immediate to see the following 

\begin{prop}
 We have isomorphisms in $\fl\Lambda$:
$$L^j_\ww\simeq e_{u_j}L^j_\ww\simeq e_{u_j}\frac{I_{u_i}\ldots I_{u_1}}{I_{u_j}\ldots I_{u_1}}\simeq \Ker(\xymatrix{M^j_\ww\ar@{->>}[r] & M^i_\ww}),$$
where $i$ is the greatest integer satisfying $u_i=u_j$ and $i<j$. (If such $i$ does not exist, then we define $M^i_\ww$ to be $0$.)
\end{prop}

Therefore the layers $L^1_\ww,\ldots,L^l_\ww$ give a filtration of the cluster-tilting object $M_\ww$.

\subsection{Mutation of tilting modules}

Let $Q$ be a finite connected quiver with vertices $\{1,\ldots ,n\}$ and without oriented cycles. 

\begin{dfa}
A basic $kQ$-module $T$ is called a \emph{tilting module} if $\Ext^1_{kQ}(T,T)=0$ and
it has $n$ non-isomorphic indecomposable summands.
\end{dfa}

For each indecomposable summand $T_i$ of $T$, it is known that there is at most one indecomposable $T_i^*\ncong T_i$ such that $(T/T_{i})\oplus T_i^*$ is a tilting module \cite{RS,Ung}, and that there is exactly one if and only if $T/T_i$ is a sincere $kQ$-module \cite{HU}. We then say that $T_i$ (and possible $T_i^*$) is a \emph{complement} for the almost complete tilting module $T/T_i$. The (possible) other complement of $T/T_i$ can be obtained using the following result:

\begin{prop}\cite{RS}\label{tilting mutation}
\begin{itemize}
\item[(a)]
If the minimal left $\add(T/T_i)$-approximation $\xymatrix{T_i\ar[r]^f & B}$ is a monomorphism, then $\Coker f$ is a complement for $T/T_i$.
\item[(b)]
If the minimal right $\add(T/T_i)$-approximation $\xymatrix{B'\ar[r]^g &T_i} $ is an epimorphism, then $\Ker g$ is a complement for $T/T_i$.
\end{itemize}
\end{prop}
There is a one-one correspondence between tilting modules $T$ and contravariantly finite torsionfree classes $\Ff=\Sub T$ containing the projective modules.

\subsection{Reflections and reflection functors}

Let $Q$ be finite quiver with vertices $\{1,\ldots ,n\}$ and without oriented cycles. Let $i\in Q_0$ be a source. 
Then the quiver $Q':=\mu_i(Q)$ is obtained by replacing all arrows starting at the vertex $i$ by arrows in the opposite direction. 

Write $kQ=P_1\oplus\cdots\oplus P_n$ where $P_j$ is the indecomposable projective $kQ$-module associated with the vertex $j$. Then using results of \cite{BGP} and \cite{APR} we have functors:
$$\xymatrix{\mod kQ\ar@<1mm>[rr]^{R_i} && \ar@<1mm>[ll]^{R_i^-}\mod kQ'}$$
where $R_i:=\Hom_{kQ}(M,-)$, $R_i^-:=-\ten_{kQ'}M$ and $M:=\tau^-P_i\oplus kQ/P_i$ which induce inverse equivalences (recall that in this paper we work with right modules) $$\xymatrix{(\mod kQ)/[e_ikQ]\ar@<1mm>[rr]^{R_i} && \ar@<1mm>[ll]^{R_i^-}(\mod kQ')/[e_i DkQ']},$$ where $\mod kQ/[e_ikQ]$ (resp. $\mod kQ'/[e_i DkQ']$) is obtained from the module category $\mod kQ$ (resp. $\mod kQ'$) by annihilating morphisms factoring through $P_i=e_ikQ$ (resp. $e_iDkQ'$). Since $i$ is a source (resp. a sink) of $Q$ (resp. $Q'$) we can regard the category $\mod kQ/[e_ikQ]$ (resp. $\mod kQ'/[e_iDkQ']$) as a full subcategory of $\mod kQ$ (resp. $\mod kQ'$).

When the vertex $i$ is not a sink or source, a reflection is still defined on the level of the Grothendieck group $K_0(\mod kQ)$. The Grothendieck group is constructed as the group with generators $[X]$ for $X\in \mod kQ$ and relations $[X]+[Z]=[Y]$ if there is a short exact sequence $\xymatrix{ X\ar@{>->}[r] & Y\ar@{->>}[r] & Z}.$ 
This is a free abelian group with basis $\{[S_1],\ldots[S_n]\}$, where $S_1,\ldots, S_n$ are the simple $kQ$-modules. With respect to this basis we define 
$$ R_i([S_j])= [S_j]+(m_{ij}-2\delta_{ij})[S_i],$$ 
where $m_{ij}$ is the number of edges of the underlying graph of $Q$ as before.

This definition is coherent with the previous one. Indeed if $i$ is a source and $M$ is an indecomposable  $kQ$-module which is not isomorphic to $P_i$, then we have $$R_i([M])=[R_i(M)].$$
 
\section{Layers associated with reduced words}

Throughout this section let $w$ be an element in the Coxeter group of an acyclic quiver $Q$, and fix $\ww=s_{u_1}\ldots s_{u_l}$ a reduced expression of $w$. For $j=1, \dots, l$ we have defined in Section 1 the layer $L^j_\ww$ as the quotient $$L^j_\ww:=\frac{I_{u_{j-1}}\ldots I_{u_1}}{I_{u_j}\ldots I_{u_1}}.$$

In this section, we investigate some main properties of these layers. We show that each layer can be seen as the image of a simple $\Lambda$-module under an autoequivalence of $\Dd^b(\fl \Lambda)$. Hence they are rigid indecomposable $\Lambda$-modules of finite length, and we compute explicitly their dimension vectors and show that they are real roots. Hence to each layer we can associate a unique indecomposable $kQ$-module with the same dimension vector \cite{Kac80}, but which is not necessarily rigid. 

Note that some of the results of this section have been proven independently in \cite{GLS} but with different proofs.

\subsection{Layers are simples up to autoequivalences}

The following easy observation is useful.

\begin{lema}\label{extend support}
Let $\widetilde{Q}$ be an acyclic quiver and $Q$ be a full subquiver of $\widetilde{Q}$. For any reduced expression $\ww$ of $w\in W_Q$, the module $M^j_\ww$ (respectively, $L^j_\ww$) for $\widetilde{Q}$ is the same as $M^j_\ww$ (respectively, $L^j_\ww$) for $Q$.
\end{lema}

\begin{proof}
Let $\Lambda:=\Lambda_Q$ and $\widetilde{\Lambda}:=\Lambda_{\widetilde{Q}}$ be the corresponding complete preprojective algebras. Let $$e:=\sum_{i\in\widetilde{Q}_0\backslash Q_0}e_i.$$ Then we have $\Lambda=\widetilde{\Lambda}/\widetilde{\Lambda} e \widetilde{\Lambda}$ and $$I_i=\frac{\widetilde{\Lambda}(1-e_i)\widetilde{\Lambda}}{\widetilde{\Lambda} e\widetilde{\Lambda}}$$
for any $i\in Q_0$. Thus the assertions follow.
\end{proof}

\begin{prop}\label{layerimagesimple}
 Let $Q$ be an acyclic quiver and $\Lambda$ the complete preprojective algebra. 
Let $\ww=s_{u_1}\ldots s_{u_l}$ be a reduced expression.
\begin{enumerate}
\item For $j=1,\ldots, l$ 
 we have isomorphisms of $\Lambda$-modules: 
$$L^j_\ww\simeq S_{u_j}\ten_\Lambda (I_{u_{j-1}}\ldots I_{u_1})\simeq S_{u_j}\ten_{\Lambda}I_{u_{j-1}}\ten_\Lambda \cdots \ten_\Lambda I_{u_1}.$$ 
\item If $Q$ is non-Dynkin, then for $j=1,\ldots, l$ 
 we have isomorphisms in $\Dd(\Mod \Lambda)$: 
$$L^j_\ww\simeq S_{u_j}\lten_\Lambda (I_{u_{j-1}}\ldots I_{u_1})\simeq S_{u_j}\lten_{\Lambda}I_{u_{j-1}}\lten_\Lambda \cdots \lten_\Lambda I_{u_1}.$$
\end{enumerate}
\end{prop}
\begin{proof}
We divide the proof into two cases, according to whether $\Lambda$ is of non-Dynkin type or of Dynkin type.

\medskip
\noindent
\textit{non-Dynkin case}:
We set $\ww':=s_{u_1}\ldots s_{u_j}$ and $\ww'':=s_{u_1}\ldots s_{u_{j-1}}$. Since $\ww''$ is reduced, by Proposition \ref{basicprop}(e) we have
$$I_{w''}\simeq I_{u_{j-1}}\ten_\Lambda\ldots\ten_\Lambda I_{u_1}\simeq I_{u_{j-1}}\lten_\Lambda\ldots\lten_\Lambda I_{u_1},$$
and hence we get the second isomorphism.

Since $\ww'=\ww''s_{u_j}$ is reduced, we have $I_{w'}=I_{u_j}I_{w''}\varsubsetneq I_{w''}$, and therefore $\Tor^\Lambda_1(S_{u_j},I_{w''})=0$ by Proposition \ref{basicprop} (d). Thus we have 
$$S_{u_j}\lten_\Lambda I_{w''}\simeq S_{u_j}\ten_\Lambda I_{w''}\simeq \frac{\Lambda}{I_{u_j}}\ten_\Lambda I_{w''}\simeq \frac{I_{w''}}{I_{u_j}I_{w''}}=L^j_\ww.$$

\medskip
\noindent
\textit{Dynkin case}:

We take a non-Dynkin quiver $\widetilde{Q}$ containing $Q$ as a full subquiver. Let $\widetilde{\Lambda}$ be the complete preprojective algebra of $\widetilde{Q}$ and $\widetilde{I}_i:=\widetilde{\Lambda}(1-e_i)\widetilde{\Lambda}$ for $i\in\widetilde{Q}_0$. Using the non-Dynkin case and Lemma \ref{extend support}, we have
$$L^j_\ww\simeq S_{u_j}\ten_{\widetilde{\Lambda}} (\widetilde{I}_{u_{j-1}}\ldots \widetilde{I}_{u_1})\simeq S_{u_j}\ten_{\widetilde{\Lambda}}\widetilde{I}_{u_{j-1}}\ten_{\widetilde{\Lambda}} \cdots \ten_{\widetilde{\Lambda}}\widetilde{I}_{u_1}.$$
For the idempotent $e:=\sum_{i\in\widetilde{Q}_0\backslash Q_0}e_i$, the twosided ideal $\widetilde{\Lambda}e\widetilde{\Lambda}$ annihilates $S_{u_j}$. Since $I_i=\widetilde{I}_i/\widetilde{\Lambda}e\widetilde{\Lambda}$ holds, we have
$$S_{u_j}\ten_{\widetilde{\Lambda}} (\widetilde{I}_{u_{j-1}}\ldots \widetilde{I}_{u_1})\simeq S_{u_j}\ten_\Lambda (I_{u_{j-1}}\ldots I_{u_1})$$
and
$$S_{u_j}\ten_{\widetilde{\Lambda}}\widetilde{I}_{u_{j-1}}\ten_{\widetilde{\Lambda}} \cdots \ten_{\widetilde{\Lambda}}\widetilde{I}_{u_1}\simeq S_{u_j}\ten_{\Lambda}I_{u_{j-1}}\ten_\Lambda \cdots \ten_\Lambda I_{u_1}.$$
Thus the assertion follows.

\end{proof}

Immediately we have the following result, which implies that $L^j_\ww$ is an indecomposable rigid $\Lambda$-module of finite length.
\begin{thma}\label{2spherical}
For $j=1,\ldots, l$ we have
\begin{itemize}
\item if $\Lambda$ is of non-Dynkin type:  
\[\dim\Ext^i_\Lambda(L_\ww^j,L_\ww^j)=\left\{\begin{array}{cc}
1&i=0,2,\\
0&\mbox{otherwise.}
\end{array}\right.\]
\item 
if $\Lambda$ is of Dynkin type:
\[\dim\Ext^i_\Lambda(L_\ww^j,L_\ww^j)=\left\{\begin{array}{cc}
1&i=0,2\ \ (\mod\ 6),\\
0& i=1\ \ (\mod\ 6).
\end{array}\right.\]
\end{itemize}
\end{thma}

Note that one can write down explicitly the dimension for the other $i$ in the Dynkin case by using $\Omega^3\simeq\nu_\Lambda$ \cite{ES98}. In the non-Dynkin case, $L^j_\ww$ is then \emph{2-spherical} in the sense of Seidel-Thomas \cite{ST}.

\begin{proof}
We divide  the proof into two cases, according to whether $\Lambda$ is of non-Dynkin type or of Dynkin type.

\medskip
\noindent
\textit{non-Dynkin case}:
By Proposition \ref{basicprop} (c), $I_{w''}$ is a tilting $\Lambda$-module with $\End_{\Lambda}(I_{w''})\simeq \Lambda$. Hence the functor $-\lten_\Lambda I_{w''}$ is an autoequivalence of $\Dd(\Mod \Lambda)$. We have $\End_\Lambda(S_j)\simeq k$ and hence $\Ext^2_\Lambda(S_j,S_j)\simeq k$ since $\Dd^b(\fl \Lambda)$ is 2-CY. Moreover since $Q$ has no loops, $\Ext^1_\Lambda(S_j,S_j)$ vanishes and since $\Lambda$ is known to have global dimension $2$, $\Ext^n_\Lambda(S_j,S_j)$ vanishes for $n\geq 3$. Hence $S_j$ is 2-spherical. Since by Proposition \ref{layerimagesimple} the layer $L^j_\ww$ is the image of the simple $S_j$ by an autoequivalence of $\Dd^b(\fl \Lambda)$, it follows that $L^j_\ww$ is also 2-spherical.

\medskip
\noindent
\textit{Dynkin case}:

We take a non-Dynkin quiver $\widetilde{Q}$ containing $Q$ as a full subquiver. Let $\widetilde{\Lambda}$ be the complete preprojective algebra of $\widetilde{Q}$. Then $\mod \Lambda$ can be seen as a full and extension closed subcategory of $\mod \widetilde{\Lambda}$.
Using the non-Dynkin case and Lemma \ref{extend support}, we get
$$\End_\Lambda(L^j_\ww)\simeq \End_{\widetilde{\Lambda}}(L^j_\ww)\simeq k \textrm{ and } \Ext^1_\Lambda(L^j_\ww,L^j_\ww)\simeq \Ext^1_{\widetilde{\Lambda}}(L^j_\ww,L^j_\ww)=0.$$ Using the fact that $\mod \Lambda$ is stably 2-CY we get $\Ext^2_\Lambda(L^j_\ww,L^j_\ww)\simeq k$.
\end{proof}

Here we state a property about two consecutive layers associated with the same vertex, which gives rise to special non-split short exact sequences in $\fl \Lambda$.

\begin{prop}\label{dimext1}
Let $1\leq i<j<k\leq l$ be integers such that $u_i=u_j=u_k$ and such that $j$ is the only integer satisfying $i<j<k$ and  $u_i=u_j=u_k$ .
Then we have $$\dim_k \Ext^1_\Lambda(L^j_\ww,L^k_\ww)=1.$$
\end{prop}

\medskip

In order to prove this proposition, we first need a lemma.
For $1\leq h \leq l$, we denote as before by $M^h_\ww$ the $\Lambda$-module $M^h_\ww:=e_{u_h}\frac{\Lambda}{I_{u_h}\ldots I_{u_1}}$.

\begin{lema}\label{petitlemme} Let $i<j<k$ be as in Proposition \ref{dimext1}. 
\begin{enumerate}
\item The map $\Hom_{\Lambda}(M^k_\ww,M^j_\ww)\rightarrow \Hom_\Lambda(M_\ww^k,M_\ww^i)$ induced by the irreducible map $M_\ww^j\rightarrow  M_\ww^i$ is an epimorphism.
\item The image of the map $\Hom_\Lambda(M_\ww^i,M_\ww^j)\rightarrow \Hom_\Lambda(M_\ww^j,M_\ww^j)$ induced by the irreducible map $M_\ww^j\rightarrow M_\ww^i$ is $\Rad_\Lambda(M_\ww^j,M_\ww^j)$ .
\end{enumerate}
\end{lema}

\begin{proof}
(1) Since $i<j<k$, then by Lemma III.1.14 of \cite{Bua2}, we have isomorphisms 
$$\Hom_{\Lambda}(M_\ww^k,M_\ww^j)\simeq e \frac{\Lambda}{I_{u_j}\ldots I_{u_1}}e \quad \textrm{and} \quad \Hom_\Lambda(M_\ww^k,M_\ww^i)\simeq 
e \frac{\Lambda}{I_{u_i}\ldots I_{u_1}}e,$$ where $e$ is the idempotent $e:=e_{u_i}=e_{u_j}=e_{u_k}$. Then the map $\Hom_{\Lambda}(M_\ww^k,M_\ww^j)\rightarrow \Hom_\Lambda(M_\ww^k,M_\ww^i)$ is the epimorphism $e \frac{\Lambda}{I_{u_j}\ldots I_{u_1}}e_{u_k}\rightarrow e \frac{\Lambda}{I_{u_i}\ldots I_{u_1}}e_{u_k}$ induced by the inclusion $I_{u_j}\ldots I_{u_1}\subset I_{u_i}\ldots I_{u_1}$.

(2) It is clear that the image is contained in the radical. 
By Lemma III.1.14 of \cite{Bua2}, we have isomorphisms
$$ \Hom_{\Lambda}(M_\ww^i,M_\ww^j)\simeq e\frac{I_{u_j}\ldots I_{u_{i+1}}}{I_{u_j}\ldots I_{u_1}}e\quad\textrm{and}\quad \Rad_\Lambda(M_\ww^j,M_\ww^j)\simeq e\frac{I_{u_j}}{I_{u_j}\ldots I_{u_1}}e.$$
The map $\Hom_\Lambda(M_\ww^i,M_\ww^j)\rightarrow \Rad_\Lambda(M_\ww^j,M_\ww^j)$ is induced by the inclusion of ideals $I_{u_j}\ldots I_{u_{i+1}}\subset I_{u_j}$.
But since $j$  is the only integer satisfying $i<j<k$ and  $u_i=u_i=u_k$ , we have $e I_{u_j}\ldots I_{u_{i+1}}e\simeq e I_{u_j}e$ and hence the map $\Hom_\Lambda(M_\ww^i,M_\ww^j)\rightarrow \Rad_\Lambda(M_\ww^j,M_\ww^j)$ is an isomorphism.
\end{proof}

\begin{proof}[Proof of Proposition \ref{dimext1}.]

By the definition of the layers, we have the following short exact sequences
$$(j)\quad \xymatrix{ L^j_\ww\ar@{>->}[r] & M_\ww^j\ar@{->>}[r] & M_\ww^i} \quad\textrm{and } (k)\quad \xymatrix{ L^k_\ww\ar@{>->}[r] & M_\ww^k\ar@{->>}[r] & M_\ww^j} $$

Let $K$ be the kernel of the composition map $M_\ww^k\rightarrow M_\ww^j\rightarrow M_\ww^i$. Then we have a short exact sequence 
$$(l)\quad \xymatrix{ K\ar@{>->}[r] & M_\ww^k\ar@{->>}[r] & M_\ww^i}$$ which gives rise to the following long exact sequence in $\mod \End_\Lambda(M_\ww)$, where $M_\ww=\bigoplus_{h=1}^l M_\ww^h$:
$$ \xymatrix{D\Ext^1_\Lambda(M_\ww^i,M_\ww)\ar[r] & D\Hom_\Lambda(K,M_\ww)\ar[r] & D\Hom_{\Lambda}(M_\ww^k, M_\ww)\ar[r] & D\Hom_{\Lambda}(M_\ww^i, M_\ww)\ar[r] & 0}$$
The space $D\Ext^1_\Lambda(M_\ww^i,M_\ww)$ is zero by Lemma III.2.1 of \cite{Bua2}, and the $\End_\Lambda(M_\ww)$-module $D\Hom_{\Lambda}(M_\ww^k, M_\ww)$ is indecomposable injective. Therefore the module $D\Hom_\Lambda(K,M_\ww)$ has simple socle, and hence $K$ is indecomposable.

Moreover from the sequences $(j)$, $(k)$ and $(l)$, we deduce that we have a short exact sequence
$\xymatrix{ L^k_\ww\ar@{>->}[r] & K\ar@{->>}[r] & L^j_\ww}$ which is non-split since $K$ is indecomposable. Hence we get 
$$\dim_k\Ext^1_\Lambda(L^j_\ww,L^k_\ww)\geq 1.$$

From $(j)$ we deduce the long exact sequence
$$\xymatrix{\cdots\ar[r] & \Hom_\Lambda(M_\ww^k,M_\ww^j)\ar[r] & \Hom_\Lambda(M_\ww^k,M_\ww^i)\ar[r] & \Ext^1_\Lambda(M_\ww^k,L_\ww^j)\ar[r] & \Ext^1_\Lambda(M_\ww^k,M_\ww^j)=0}.$$
Hence by Lemma \ref{petitlemme} (1) we get  $\Ext^1_\Lambda(M_\ww^k,L_\ww^j)=0$.

From $(j)$ we also deduce the long exact sequence
$$\xymatrix{0\ar[r] & \Hom_\Lambda(M_\ww^i,M_\ww^j)\ar[r] & \Hom_\Lambda(M_\ww^j,M_\ww^j)\ar[r] & \Hom_\Lambda(L^j_\ww,M_\ww^j)\ar[r] & \Ext^1_\Lambda(M_\ww^i,M_\ww^j)=0}.$$
Hence by Lemma \ref{petitlemme} (2) we get $\Hom_\Lambda(L^j_\ww,M_\ww^j)\simeq \Hom_\Lambda(M_\ww^j,M_\ww^j)/\Rad_\Lambda(M_\ww^j,M_\ww^j)$ which is one dimensional since $M_\ww^j$ is indecomposable.

Finally using $(k)$ we get the long exact sequence $$\xymatrix{\cdots\ar[r] & \Ext^1_\Lambda(M_\ww^k,L^j_\ww)\ar[r] & \Ext^1_\Lambda(L^k_\ww,L^j_\ww)\ar[r] & \Ext^2_\Lambda(M_\ww^j,L^j_\ww)\ar[r] & \cdots}$$
By the 2-CY property and the previous remarks we have $$\Ext^1_\Lambda(M_\ww^k,L_\ww^j)=0 \quad \textrm{and} \quad\Ext^2_\Lambda(M_\ww^j,L^j_\ww)\simeq D\Hom_\Lambda(L^j_\ww,M_\ww^j)\simeq k$$ and therefore 
$$\dim_k\Ext^1_\Lambda(L^j_\ww,L^k_\ww)\leq 1.$$

\end{proof}

\subsection{The dimension vectors of the layers}
In this section we investigate the action of the functor $-\lten_\Lambda I_w$ at the level of the Grothendieck group of $\Dd^b(\fl \Lambda)$ when $\Lambda$ is not of Dynkin type. We show that this action has interesting connections with known actions of Coxeter groups. We denote by $[-\lten_\Lambda I_w]$ the induced automorphism of $K_0(\Dd^b(\fl\Lambda))$.

\begin{lema}\label{dimvectlemma} Let $Q$ be a non-Dynkin quiver.
For all $i,j$ in $Q_0$ we have $$[S_j\lten_\Lambda I_i]=[S_j] + (m_{ij} -2 \delta_{ij})[S_i]$$ in $K_0(\Dd^b(\fl \Lambda))$, where $m_{ij}$ is the number of arrows between $i$ and $j$ in $Q$.
\end{lema}
\begin{proof}
Since $S_i=\Lambda/I_i$, we have $DS_i\simeq S_i$ as $\Lambda$-bimodules. Hence we have the following isomorphisms in $\Mod (\Lambda^{op}\ten \Lambda)$:
$$\begin{array}{rcl}S_j\lten_\Lambda S_i
&\simeq& D\Hom_k(S_j\lten_\Lambda S_i,k)\\ & \simeq &DR\Hom_\Lambda(S_j, \Hom_k(S_i,k))\\
& \simeq &DR\Hom_\Lambda(S_j, DS_i)\\ & \simeq&
DR\Hom_\Lambda(S_j,S_i).\end{array}$$
Therefore we have $$[S_j\lten_\Lambda S_i]= (\sum_{t}(-1)^t\dim\Ext^t_\Lambda(S_j,S_i))[S_i]=(2\delta_{ij}-m_{ij})[S_i].$$

From the triangle $\xymatrix{S_i[-1]\ar[r] & I_i\ar[r] & \Lambda\ar[r] & S_i}$ we get a triangle
$$\xymatrix{S_j\lten_\Lambda S_i[-1]\ar[r] &S_j\lten_\Lambda  I_i\ar[r] & S_j \ar[r] & S_j\lten_\Lambda S_i}.$$
Hence we have $[S_j\lten_\Lambda I_i]=[S_j]-[S_j\lten_\Lambda S_i]=[S_j]-(2\delta_{ij}-m_{ij})[S_i].$
\end{proof}

From Lemma \ref{dimvectlemma}, we deduce the following results.

\begin{thma}\label{realroot}
 Let $\Lambda$ be the complete preprojective algebra of \emph{any} type.
\begin{enumerate}
\item For $j=1,\ldots, l$ we have $[L^j_\ww]=R_{u_1}\ldots R_{u_{j-1}}([S_{u_j}])$, where the $R_t$ are the reflections defined in Section 1.
In particular all $[L^j_\ww]$ are positive real roots.

\item $[L^1_\ww],\ldots,[L^l_\ww]$ are pairwise different in $K_0(\Dd^b(\fl \Lambda))$.

\item For $j=1,\ldots,l$, there exists a unique indecomposable $kQ$-module $(L^j_\ww)_{Q}$ such that $[L^j_\ww]=[(L^j_\ww)_Q]$.
\end{enumerate}
\end{thma}
\begin{proof}
(1)  As in the previous subsection we treat separately the Dynkin and the non-Dynkin case.
The non-Dynkin case is a direct consequence of Lemma \ref{dimvectlemma} and Proposition \ref{layerimagesimple}.

For the Dynkin case, we can follow the strategy of the proof of Theorem \ref{2spherical} . We introduce an extended Dynkin quiver containing $Q$ as subquiver. Then applying reflection functors associated to the vertices of $Q$ to modules whose support do not contain the additional vertex is the same as applying the reflection functors of $Q$. Thus using Lemma \ref{extend support}, the equality coming from the non-Dynkin quiver gives us the equality for $Q$.

Hence the $[L^j_{\ww}]$ are real roots, which are clearly positive since $L^j_{\ww}$ is a module.

(2) By \cite[Prop. 4.4.4]{BB}, $[S_{u_1}]$, $R_{u_1}([S_{u_2}]),\ldots,R_{u_1}\ldots R_{u_{l-1}}([S_{u_l}])$ are pairwise different.
Thus the assertion follows from (1).

(3) From (1) we know that the dimension vector of the layer $L^j_\ww$ is a positive real root, and we get the result applying Kac's Theorem \cite{Kac80}. 
\end{proof}

The layer $L^j_\ww$ is always rigid as a $\Lambda$-module, but the associated indecomposable $kQ$-module $(L^j_\ww)_Q$ is not always rigid as shown in the following.
\begin{exa}\label{exampleLQnonrigid}
Let $Q$ be the quiver $\xymatrix@-.5cm@R=-1mm{ &2\ar[dr] & \\ 1\ar[ur]\ar[rr] && 3}$, and $\ww:=s_1s_2s_3s_2s_1s_3$. Then we have $$L^1_\ww={\bsm 1\esm},\quad L^2_\ww={\bsm 2\\1\esm},\quad L^3_\ww={\bsm &3&&\\1&&2&\\ &&&1\esm}, \quad L^4_\ww={\bsm 3\\1\esm},\quad L^5_\ww={\bsm &&2&&3&&\\&3&&1&&2&\\1&&&&&&1\esm},\quad \textrm{and} \quad L^6_\ww={\bsm &&2&&3\\&3&&1&\\1&&&&\esm}.$$
Thus the associated indecomposable $kQ$-modules are the following:
$$(L^j_\ww)_Q=L^j_\ww \textrm{ for } j=1,\ldots 4,\quad (L^5_\ww)_Q={\bsm &3&&&&&\\ 1&&2&&3&&\\&&&1&&2&\\&&&&&&1\esm},\quad \textrm{and} \quad (L^6_\ww)_Q={\bsm  &3&&&\\1&&2&&3\\&&&1&\esm}.$$
The module $(L^6_\ww)_Q$ lies in the tube of rank 2, with indecomposable objects ${\bsm 3\\1\esm}$ and ${\bsm 2\esm}$ on the border of the tube. Since $(L^6_\ww)_Q$ is not on the border of the tube, it is not rigid.
\end{exa}

\begin{dfa}\cite{BB}
Let $Q$ be an acyclic quiver with $n$ vertices, and $W_Q$ be the Coxeter group of $Q$.
Let $V$ be the vector space with basis $v_1,\ldots,v_n$. The \emph{geometric representation} $W_Q\rightarrow \textrm{GL}(V)$ of $W_Q$ is defined by reflections 
$$s_iv_j:=v_j+(m_{ij}-2\delta_{ij})v_i.$$
The \emph{contragradient of the geometric representation} $W_Q\rightarrow \textrm{GL}(V)$ is then
$$s_iv_j^*=\left\{\begin{array}{cl} v_j^* & i\neq j\\ -v_j^*+\sum_{t\neq j}m_{tj}v_t^* & i=j\end{array}\right.$$
\end{dfa} 

The Grothendieck group $K_0(\Dd^b(\fl \Lambda))$ has a basis consisting of the simple $\Lambda$-modules, and $K_0(\Kk^b(\proj\Lambda))$ has a basis consisting of the indecomposable projective $\Lambda$-modules. 
\begin{prop}
Let $\Lambda$ be the complete preprojective algebra of non-Dynkin type.
\begin{enumerate}
\item The Coxeter group $W_Q$ acts on $K_0(\Dd^b(\fl \Lambda))$ by $w\mapsto [-\lten_\Lambda I_w]$ as the geometric representation.
\item The Coxeter group $W_Q$ acts on $K_0(\mathcal{K}^b(\proj \Lambda))$ by $w\mapsto [-\lten_\Lambda I_w]$ as the contragradient of the geometric representation.
\end{enumerate}
\end{prop}
\begin{proof}
(1) This follows directly from Lemma \ref{dimvectlemma}.

(2) This is shown in \cite[Theorem 6.6]{IR}. It is assumed in \cite{IR} that $Q$ is extended Dynkin, but this assumption is not used in the proof for this statement.
\end{proof}

\subsection{Reflection functors and ideals $I_i$}

In this subsection, we state some basic properties of the first layers. In particular we show that the equivalence $-\lten_\Lambda I_i$, when $Q$ is not Dynkin, can be interpreted as a reflection functor of the category $\Dd^b(\fl\Lambda)$.

\begin{lema}\label{lemmelayer1}
 Let $Q$ be an acyclic quiver, and $\Lambda=\Lambda_Q$. Let $c\in W_Q$ be a Coxeter element admissible with respect to the orientation of $Q$. Let $i\in Q_0$ be a source of $Q$ and $R_i^{-}:\mod kQ\to\mod kQ'$ be the reflection functor for $Q':=\mu_i(Q)$.
Then we have the following :
\begin{enumerate}
 \item $\Lambda/I_c\simeq kQ$  in $\mod \Lambda$. So we view $kQ$-modules as $\Lambda$-modules annihilated by $I_c$.

\item If $\ww'$ is a subsequence of $\ww$, and if $\ww$ is a subsequence of $c\ww'$ (where $\ww'$, $c\ww'$ and $\ww$ are reduced expressions in $W_Q$), then $I_{\ww'}/I_\ww$ is a $kQ$-module.

\item $I_i/I_{cs_i}$ is a $kQ'{}^{op}\ten kQ$-module and isomorphic to $\tau^{-1}P_i\oplus kQ/P_i=R_i^{-}(kQ)$ as a $kQ'{}^{op}\ten kQ$-module, where $P_i=e_ikQ$ is the indecomposable projective $kQ$-module associated to $i$ and $\tau$ is the AR-translation of $\mod kQ$.

\end{enumerate}

\end{lema}

\begin{proof}
(1) This is Propositions II.3.2 and II.3.3 of \cite{Bua2}.

(2) Since $I_{\ww'}I_c=I_{c\ww'}\supset I_\ww$, we have that $I_{\ww'}/I_\ww$ is annihilated by $I_c$.

(3)  Note that by Proposition \ref{basicprop} (b) we have $e_jI_i=e_j\Lambda$ and $e_jI_{cs_i}=e_jI_iI_c=e_jI_c$ if $j\neq i$. Therefore by (1) it is enough to prove that $e_i I_i/I_{cs_i}\simeq \tau^{-1}(e_i kQ)$.

The projective resolution of $e_iI_i$ in $\mod \Lambda$ has the form:
$$\xymatrix{(*) &e_i\Lambda\ar[r] & \bigoplus_{a\in \bar{Q}_1, s(a)=i}e_{t(a)}\Lambda\ar[r] & e_i I_i\ar[r] & 0.}$$

Applying the functor $-\ten_{\Lambda}\dfrac{\Lambda}{I_c}$ to the exact sequence $(*)$, we get an exact sequence
$$(**)\xymatrix{e_i\dfrac{\Lambda}{I_c}\ar[r] & \bigoplus_{a\in \bar{Q}_1, s(a)=i}e_{t(a)}\dfrac{\Lambda}{I_c}\ar[r] & e_i\dfrac{I_i}{I_iI_{c}}=e_i\dfrac{I_i}{I_{cs_i}}\ar[r] & 0}.$$
Since $i$ is a source in $Q$, we have the set equality $$\{a\in \bar{Q}_1, \textrm{ with } s(a)=i\}=\{a\in Q_1, \textrm{ with } s(a)=i\}.$$ Therefore by (1) the  exact sequence $(**)$ is 
$$\xymatrix{e_i kQ\ar[r] & \bigoplus_{a\in Q_1, s(a)=i}e_{t(a)}kQ\ar[r] & e_i\dfrac{I_i}{I_{cs_i}}\ar[r] & 0}.$$ Hence we have $e_i\dfrac{I_i}{I_{cs_i}}\simeq \tau^{-}(e_i kQ)$.

\end{proof}
 
From Lemma \ref{lemmelayer1} we deduce the following result which gives another interpretation of the reflection functor.

\begin{cora}\label{commutativediagram}
 Let $Q$ be an acyclic quiver  and $\Lambda=\Lambda_Q$.
\begin{enumerate}
\item Let $i\in Q_0$ be a sink of $Q$. Let $Q':=\mu_i(Q)$. Then the following diagram commutes
$$\xymatrix@C=1.5cm{\mod kQ/ [e_iDkQ] \ar[r]^{R_i^{-}} \ar@{^(->}[d] & \mod kQ'/ [e_i kQ']\ar@{^(->}[d]\\ \fl \Lambda\ar[r]_{-\ten_{\Lambda}I_i} & \fl \Lambda},$$ where the vertical functors are the natural inclusions.
If $Q$ is not Dynkin, then the following diagram commutes
$$\xymatrix@C=1.5cm{\mod kQ/ [e_iDkQ] \ar[r]^{R_i^{-}} \ar@{^(->}[d] & \mod kQ'/ [e_i kQ']\ar@{^(->}[d]\\ \Dd^b(\fl \Lambda)\ar[r]_{-\lten_{\Lambda}I_i} & \Dd^b(\fl \Lambda)},$$ where the vertical functors are the natural inclusions.

\item Let $c=s_{u_1}\cdots s_{u_n}$ be a Coxeter element and
$C^{-}:=R_{u_n}^{-}\circ\cdots\circ R_{u_1}^{-}:\mod kQ\to\mod kQ$ the Coxeter functor. Then the following diagram commutes
$$\xymatrix@C=1.5cm{\mod kQ/ [DkQ] \ar[r]^{C^{-}} \ar@{^(->}[d] & \mod kQ/ [kQ]\ar@{^(->}[d]\\ \fl \Lambda\ar[r]_{-\ten_{\Lambda}I_c} & \fl \Lambda},$$ where the vertical functors are the natural inclusions.
If $Q$ is not Dynkin, then the following diagram commutes
$$\xymatrix@C=1.5cm{\mod kQ/ [DkQ] \ar[r]^{C^{-}} \ar@{^(->}[d] & \mod kQ/ [kQ]\ar@{^(->}[d]\\ \Dd^b(\fl \Lambda)\ar[r]_{-\lten_{\Lambda}I_c} & \Dd^b(\fl \Lambda),}$$ where the vertical functors are the natural inclusions.
In particular we have $I_{c^l}/I_{c^{l+1}}\simeq \tau^{-l}(kQ).$ 
 \end{enumerate}

\end{cora}

\begin{proof} 
(1) Denote by $c$ the Coxeter element admissible with respect to the orientation of $Q$, and by $c'=s_ics_i$ the Coxeter element admissible with respect to the orientation of $Q'$.
We have the following isomorphisms in $\fl \Lambda$. 
$$\begin{array}{rcll} kQ\ten_\Lambda I_i& \simeq& \Lambda/I_c\ten_\Lambda I_i &\textrm{by Lemma \ref{lemmelayer1} (1)}\\ & \simeq & I_i/I_cI_i \simeq I_i/I_i I_{c'} & \\ & \simeq & \tau^{-1}P_i\oplus kQ/P_i& \textrm{ by Lemma \ref{lemmelayer1} (3)}\end{array}$$
Thus on $\mod kQ/ [e_iDkQ]$, we have $-\ten_\Lambda I_i=-\ten_{kQ}(kQ\ten_\Lambda I_i)\simeq-\ten_{kQ}(\tau^{-1}P_i\oplus kQ/P_i)=R_i^{-}$.

The latter assertion can be shown quite similarly since we have $kQ\lten_\Lambda I_i\simeq kQ\ten_\Lambda I_i$ by Proposition \ref{layerimagesimple}.

(2) This is a direct consequence of (1).
\end{proof}

\section{Tilting modules and $c$-sortable words}

In this section $Q$ is a finite acyclic quiver, $\Lambda$ is the complete preprojective algebra associated with $Q$ and $c$ a Coxeter element admissible with respect to the orientation of $Q$. The purpose of this section is to investigate the layers for words $\ww$ satisfying a certain property called $c$-sortable.  

\begin{dfa}\cite{Rea}
 Let $c$ be a Coxeter element of the Coxeter group $W_Q$. 
Usually we fix a reduced expression of $c$ and regard $c$ as a reduced word.
An element $w$ of $W_Q$ is called \emph{c-sortable} if there exists a reduced expression $\ww$ of $w$ of the form $\ww=c^{(0)}c^{(1)}\ldots c^{(m)}$ where all $c^{(t)}$ are subwords of $c$ whose supports satisfy 
$$supp(c^{(m)})\subseteq supp(c^{(m-1)})\subseteq \ldots \subseteq supp(c^{(1)})\subseteq supp(c^{(0)})\subseteq Q_0.$$
For $i\in Q_0$, if $s_i$ is in the support of $c^{(t)}$, by abuse of notation, we will write $i\in c^{(t)}$.
\end{dfa}

Then $c$-sortability does not depend on the choice of reduced expression of $c$.
It is immediate that the expression $\ww=c^{(0)}c^{(1)}\ldots c^{(m)}$ is unique for any $c$-sortable element of $W_Q$ \cite{Rea}.

Let $w$ be an element of $W_Q$, and $\ww=s_{u_1}\ldots s_{u_l}$ a reduced expression. Recall from Section 1 that for $j=1, \ldots, l$ the layer $L^j_\ww$ is defined to be the $\Lambda$-module:
$$ L_\ww^j= e_{u_j}\frac{I_{u_k}\ldots I_{u_1}}{I_{u_j}\ldots
  I_{u_1}}=\frac{I_{u_{j-1}}\ldots I_{u_{1}}}{I_{u_j}\ldots I_{u_1}}$$
where $k<j$ satisfies $u_k=u_j$ and is maximal with this property.

\begin{exa}\label{calculation of L and M}
Let $Q$ be the quiver $\xymatrix@R=.1cm@C=.5cm{&2\ar[dr] &
  \\1\ar[rr]\ar[ur] &&3}$ and $\ww:=s_1s_2s_3s_1s_2s_1$.

The standard cluster-tilting object $M_\ww$ in $\Sub \Lambda_w$ has the
following indecomposable direct summands

\begin{eqnarray*}M^1_\ww ={\bsm 1\esm},\ M^2_\ww={\bsm 2\\1\esm},\ M^3_\ww={\bsm
 &3&&\\1&&2&\\&&&1\esm},\ M^4_\ww={\bsm
&1&&&\\2&&3&&\\&1&&2&\\&&&&1\esm},\ M^5_\ww={\bsm
&&2&&&&\\&3&&1&&&\\1&&2&&3&&\\&&&1&&2&\\&&&&&&1\esm},\ M^6_\ww={\bsm
  &&&1&&&\\&&2&&3&&\\&3&&1&&2&\\1&&&&&&1\esm}.\end{eqnarray*}

Then we can easily compute the layers $L^1_\ww,\ldots, L^6_\ww$. They are the indecomposable summands of the $M^i_\ww$ as $kQ$-modules: 

\begin{eqnarray*}L^1_\ww={\bsm 1\esm},\quad  L^2_\ww={\bsm 2\\1\esm}, \quad L_\ww^3={\bsm
 &3&&\\1&&2&\\&&&1\esm},\quad  L_\ww^4={\bsm
2&&3&&\\&1&&2&\\&&&&1\esm},\quad  L_\ww^5={\bsm
\\&3&&&&&\\1&&2&&3&&\\&&&1&&2&\\&&&&&&1\esm},
\quad \textrm{and}\quad L^6_\ww={\bsm
 3\\1\esm}.\end{eqnarray*}
\end{exa}

Here is a nice characterization of $c$-sortable words.

\begin{thma}
 Let $w$ be an element of $W_Q$ and $\ww=s_{u_1} s_{u_2}\ldots s_{u_l}$ be a reduced expression of $w$. Then we have the following:
\begin{enumerate}
 \item if there exists a Coxeter element $c$ such that $w$ is $c$-sortable and $\ww$ is the $c$-sortable expression of $w$, then $L^j_\ww$ is in $\mod kQ$ for all $j=1,\ldots,l$, where $Q$ is admissible for the Coxeter element $c$;
\item if for all $j=1,\ldots,l$ the layer $L^j_\ww$ is in $\mod kQ$ for a certain orientation of $Q$, then $w$ is $c$-sortable, where $c$ is the  Coxeter element admissible for the orientation of $Q$.
\end{enumerate}
\end{thma}

\begin{proof}
(1)  Assume that $\ww=s_{u_1}\ldots s_{u_l}$ is a $c$-sortable word. Let $j\geq 1$, and $k$ be the (possibly) last index such that $u_j=u_k$ and $k<j$. Since $\ww$ is $c$-sortable, the word $s_{u_1}\ldots s_{u_j}$ is a subsequence of $cs_{u_1}\ldots s_{u_k}$.
Therefore we have that ${\displaystyle L^j_\ww=e_{u_j}\frac{I_{u_k}\ldots I_{u_1}}{I_{u_j}\ldots I_{u_1}}}$
is a $kQ$-module by Lemma \ref{lemmelayer1}(2).


(2) We prove this assertion by induction on the length of the word $w$. For $l(w)=1$ the result is immediate.
By Lemma \ref{extend support} we can assume that the support of $\ww$ is $Q_0$.

Assume that (2) is true for any word $\ww$ of length $\leq l-1$ and let $\ww:=s_{u_1}\ldots s_{u_l}$ be a reduced expression such that $L^j_\ww$ is a $kQ$-module for all $j=1,\ldots, l$. We first show that $u_1$ is a source of $Q$. Assume it is not, then there exists $k\geq 2$ such that there is an arrow $u_k\rightarrow u_1$ in $Q$. Take the smallest such number. It is then not hard to check that the top of $L^k_\ww$ is the simple $S_{u_k}$ and that the kernel of the map $L^k_\ww\rightarrow S_{u_k}$ contains $S_{u_1}$ in its top. Thus $L^k_\ww$ is not a $kQ$-module, which is a contradiction. Therefore $u_1$ is a source of $Q$. 

We have
$L^j_\ww=L^{j-1}_{\ww'}\ten_\Lambda I_{u_1}$ for $j=2,\ldots, l$ by Proposition~\ref{layerimagesimple}, where $\ww':=s_{u_2}\ldots s_{u_l}$. By Theorem \ref{realroot} (1) we have $[L^j_\ww]=R_{u_1}\circ \ldots \circ R_{u_{j-1}}([S_{u_j}])$ in the Grothendieck group $K_0(\Dd^b(\fl\Lambda))$. By Theorem \ref{realroot} (2) we then have $[L^j_\ww]\neq [S_{u_1}]$ for $j\geq 2$. Thus $L^j_\ww$ is not isomorphic to the simple projective $e_{u_1}kQ=S_{u_1}$ if $j\geq2$. Then by Corollary \ref{commutativediagram}, we get $$L^{j-1}_{\ww'}\simeq R_{u_1}(L^j_\ww)\in \mod kQ'/[e_{u_1}DkQ']$$ where $Q'=\mu_{u_1}(Q)$. By the induction hypothesis we get that $\ww'$ is $c'$-sortable where $c'$ is the Coxeter element admissible for the orientation of $Q'$, \emph{i.e.}  $c'=s_{u_1}cs_{u_1}$.   
We get the conclusion using the following criterion which detects $c$-sortability:

\begin{lema}\cite[Lemma 2.1]{Rea}
 Let $c:=s_{u_1}\ldots s_{u_n}$ be a Coxeter element. If $l(s_{u_1}w)<l(w)$, then $w$ is $c$-sortable if and only if $s_{u_1}w$ is $s_{u_1}cs_{u_1}$-sortable. 
\end{lema}

\end{proof}

\subsection{Comparison of three series of $kQ$-modules}

To the $c$-sortable word $\ww=s_{u_1}\ldots s_{u_l}=c^{(0)}c^{(1)}\ldots c^{(m)}$, we associate two series of $kQ$-modules $T^j_\ww$ and $U^j_\ww$, and show that they coincide with $L^j_\ww$.

For $j=1,\ldots, l$, we define $kQ^{(0)}$-modules $T_\ww^j$. For $1\leq
j\leq l(c^{(0)})$, $T^j_\ww$ is the projective $kQ^{(0)}$-module $e_{u_j}kQ$. For
$j>l(c^{(0)})$, let $k$ be the maximal integer such that $k<j$ and
$u_k=u_j$. We define $T^j_\ww$ as the cokernel of the map
$$ f^j_\ww:T^k_\ww\rightarrow E$$
where $f^j_\ww$ is a minimal left $\add\{T^{k+1}_\ww,\ldots,T_\ww^{j-1}\}$-approximation.

\begin{exa}\label{calculation of T}
Let $Q$ be the quiver $\xymatrix@R=.1cm@C=.5cm{&2\ar[dr] &
  \\1\ar[rr]\ar[ur] &&3}$ and $\ww:=s_1s_2s_3s_1s_2s_1$
in Example \ref{calculation of L and M}.

Let us compute the $T^j_\ww$. For $j\leq 3$ the $T^j_\ww$ are the
projective $kQ$-modules, thus we have 
\begin{eqnarray*}T^1_\ww={\bsm 1\esm},\quad  T^2_\ww={\bsm 2\\1\esm},
  \quad \textrm{and}\quad T_\ww^3={\bsm
 &3&&\\1&&2&\\&&&1\esm}.\end{eqnarray*}
Then we have to compute approximations. We have a short exact sequence
$$\xymatrix{0\ar[r] & {\bsm 1\esm}\ar[r] & {\bsm 2\\1\esm}\oplus{\bsm
    &3&&\\1&&2&\\&&&1\esm}\ar[r] & {\bsm
2&&3&&\\&1&&2&\\&&&&1\esm} \ar[r] & 0},$$
where the map $\xymatrix{{\bsm 1\esm}\ar[r] & {\bsm 2\\1\esm}\oplus{\bsm
    &3&&\\1&&2&\\&&&1\esm}}$ is the minimal left $\add(T^2_\ww\oplus T^3_\ww)$-approximation of $T^1_\ww$.
Hence we have $T^4_\ww={\bsm
2&&3&&\\&1&&2&\\&&&&1\esm}$. We have an exact sequence
$$\xymatrix{0\ar[r] & {\bsm 2\\1\esm}\ar[r] & {\bsm
    &3&&\\1&&2&\\&&&1\esm}\oplus{\bsm
2&&3&&\\&1&&2&\\&&&&1\esm}\ar[r] & {\bsm
\\&3&&&&&\\1&&2&&3&&\\&&&1&&2&\\&&&&&&1\esm}\ar[r] & 0},$$
where $\xymatrix{{\bsm 2\\1\esm}\ar[r] & {\bsm
    &3&&\\1&&2&\\&&&1\esm}\oplus{\bsm
2&&3&&\\&1&&2&\\&&&&1\esm}}$ is the minimal left $\add(T^3_\ww\oplus T^4_\ww)$-approximation of $T^2_\ww$. Hence we have $T^5_\ww={\bsm
\\&3&&&&&\\1&&2&&3&&\\&&&1&&2&\\&&&&&&1\esm}$. There is an exact
sequence
$$\xymatrix{0\ar[r] & {\bsm
2&&3&&\\&1&&2&\\&&&&1\esm}\ar[r] & {\bsm
\\&3&&&&&\\1&&2&&3&&\\&&&1&&2&\\&&&&&&1\esm}\ar[r] & {\bsm
3\\1\esm}\ar[r] & 0},$$ hence $T^6_\ww={\bsm
3\\1\esm}$. So we have $T^j_\ww=L^j_\ww$ for any $j$.
\end{exa}

To define the $kQ$-modules $U^j_\ww$, the following notion is convenient.

\begin{dfa}
An \emph{admissible triple} is a triple $(Q,c,\ww)$ consisting of an
acyclic quiver $Q$, a Coxeter element $c$ admissible with respect to the
orientation of $Q$, and a $c$-sortable word $\ww=c^{(0)}c^{(1)}\ldots c^{(m)}$. 

We denote by $Q^{(j)}$ the quiver $Q$ restricted to the support of $c^{(j)}$.
\end{dfa}

\medskip 
Let $(Q,c,\ww)$ be an admissible triple with $\ww=s_{u_1}s_{u_2}\ldots
s_{u_l}$. For $j=1,\ldots, l$, we define $kQ$-modules $U_\ww^j$ by induction on $l$. 

If $l=1$ then we define $U^1_\ww=e_{u_1}kQ^{(0)}$, the simple projective $kQ^{(0)}$-module associated to the vertex $u_1$. 

Assume $l\geq 2$. Then we write $\ww=s_{u_1}\ww'$.
It is not hard to check that the triple
\[(Q'=\mu_{u_1}(Q^{(0)}), s_{u_1}c^{(0)}s_{u_1}, \ww')\]
is an admissible triple with $l(w')=l-1$. Therefore by the induction hypothesis we have $kQ'$-modules $U^1_{\ww'},\ldots U^{l-1}_{\ww'}$. For $j=2,\ldots, l$ we define 
$$U^j_\ww =R_{u_1}^{-}(U^{j-1}_{\ww'})$$
where $R_{u_1}^{-}$ is the reflection functor
$$ \xymatrix{\mod kQ'=\mod k(\mu_{u_1}Q^{(0)})\ar[rr]^(.6){R_{u_1}^{-}} && \mod kQ^{(0)}}$$ 
at the source $u_1$ of $Q^{(0)}$.

\begin{exa}\label{calculation of U}
Let $Q$ be the quiver $\xymatrix@R=.1cm@C=.5cm{&2\ar[dr] &
  \\1\ar[rr]\ar[ur] &&3}$ and $\ww:=s_1s_2s_3s_1s_2s_1$
in Examples \ref{calculation of L and M} and~\ref{calculation of T}.

Let us now compute the $U^j_\ww$'s. By definition $U^1_\ww={\bsm 1\esm}$. Then we have 
\begin{eqnarray*} U^2_\ww=R^-_1({\bsm 2\esm})= {\bsm 2\\1\esm},\quad U^3_\ww=R^-_1R^-_2({\bsm 3\esm})=R^-_1({\bsm 3\\2\esm})={\bsm &3&&\\1&&2&\\&&&1\esm}\\\textrm{and}\quad U^4_\ww=R^-_1R^-_2R^-_3({\bsm 1\esm})=R_1^-R_2^-({\bsm 1\\3\esm})=R^-_1({\bsm &1&&\\2&&3&\\&&&2\esm})={\bsm 2&&3&&\\&1&&2&\\&&&&1\esm}.\end{eqnarray*}  
And finally we have $U^5_\ww=R^-_1R^-_2R^-_3\tilde{R}^-_1({\bsm 2\esm})$ and
$U^6_\ww=R^-_1R^-_2R^-_3\tilde{R}^-_1\tilde{R}^-_2({\bsm 1\esm})$ where $\tilde{R}_i^-$ is the reflection functor associated to the quiver $\xymatrix{1\ar[r]& 2}$. Therefore we have
\begin{eqnarray*}
U^5_\ww &=& R_1^-R_2^-R_3^-({\bsm 2\\1\esm})= R^-_1R^-_2({\bsm
 &2&&\\3&&1&\\&&&3\esm})=R^-_1({\bsm
3&&1&&\\&2&&3&\\&&&&2\esm})={\bsm
\\&3&&&&&\\1&&2&&3&&\\&&&1&&2&\\&&&&&&1\esm}\\
U^6_\ww &=& R_1^-R_2^-R_3^-\tilde{R}^-_1({\bsm 1\\2\esm})= R^-_1R^-_2R^-_3({\bsm 2\esm})= R^-_1R^-_2({\bsm 2\\3\esm})=R^-_1({\bsm 3\esm})={\bsm 3\\1\esm}.\end{eqnarray*}
So we have $U^j_\ww=L^j_\ww=T^j_\ww$ for any $j$.
\end{exa}

\begin{thma}\label{LequalT}
Let $\ww=s_{u_1}\ldots s_{u_l}$ be a $c$-sortable word where $c$ is admissible for the orientation of $Q$.
\begin{enumerate}
\item We have $ L^j_\ww\simeq U^j_\ww$ for $j=1,\ldots,l$.
\item We have $ U^j_\ww\simeq T^j_\ww$ for $j=1,\ldots,l$ and $f^j_\ww$ is a monomorphism for $j=l(c^{(0)})+1,\ldots,l$.
\end{enumerate}
\end{thma}

\begin{proof}
(1) By definition $L^1_\ww= e_{u_1}\Lambda/I_{u_1}=S_{u_1}$ and $U^1_\ww=e_{u_1}kQ^{(0)}=S_{u_1}$. Hence we get $U^1_\ww=L^1_\ww$.

Let $\ww'$ be the word $s_{u_2}\ldots s_{u_l}$. We will prove that 
 $L^j_\ww=R_{u_1}^{-}(L^{j-1}_{\ww'})$ for $j\geq 2$.

By Lemma \ref{lemmelayer1} (3) we have $R^{-}_{u_1}(-)=
 -\otimes_{kQ'}\frac{I_{u_1}}{I_{c'}I_{u_1}}$. We can write 
$$L^{j-1}_{\ww'}=\frac{e_{u_j}I_{u_k}\ldots I_{u_2}}{e_{u_j}I_{u_j}\ldots
  I_{u_2}}=: \frac{Y}{X}.$$

We have the following exact sequence:

$$\xymatrix{\frac{Y}{X}\ten_{\Lambda}I_{c'}I_{u_1}\ar^{a}[r]&
  \frac{Y}{X}\ten_{\Lambda}I_{u_1} \ar^{b}[r]&
  \frac{Y}{X}\ten_{\Lambda}\frac{I_{u_1}}{I_{c'}I_{u_1}} \ar[r]& 0}$$

By Proposition \ref{basicprop} (b) we have the inclusion $YI_{c'}\subset X$ since $u_2\cdots u_j$ is a
subword of $c'u_2\ldots u_k$. Thus we have $a(y\otimes pq)=yp\otimes q=0$
for any $y\in\frac{Y}{X}$, $p\in I_{c'}$ and $q\in I_{u_1}$.
Thus $a=0$ and $b$ is an isomorphism.
We have isomorphisms
$$R_{u_1}^{-}(L_{\ww'}^{j-1})=\frac{Y}{X}\ten_{\Lambda}\frac{I_{u_1}}{I_{c'}I_{u_1}}
\stackrel{b}{\simeq} \frac{Y}{X}\ten_{\Lambda}I_{u_1} \simeq
\frac{Y\ten_\Lambda I_{u_1}}{\Imm(X\ten_\Lambda I_{u_1}\to Y\ten_\Lambda I_{u_1})}\simeq
\frac{Y\ten_\Lambda I_{u_1}}{X\ten_\Lambda I_{u_1}}\simeq L^j_\ww.$$

\medskip
(2) We will now prove that $U^j_\ww\simeq T^j_\ww$. For $j\leq l(c^{(0)})$ this is
clear because of a basic property of reflection functors.

Assume $j> l(c^{(0)})$. Let $k$ be the maximal integer such that
$u_k=u_j$ and $k<j$. It exists because $j>l(c^{(0)})$ and $\ww$ is
$c$-sortable. We define the subwords 
$\ww''=s_{u_1}\ldots s_{u_{k-1}}$ and $\ww'=s_{u_{k}}\ldots s_{u_j}$ of $\ww$. Let
$c'$ be $s_{u_k}\ldots s_{u_{j-1}}$, and $Q'$ be the quiver
$\mu_{u_{k-1}}\circ \cdots \circ \mu_{u_{1}}(Q)$. Then $(Q',c',\ww')$ is an
admissible triple. We have $U_{\ww'}^1=S_{u_k}$ and
$U_{\ww'}^{j-k+1}=R_{c'}^{-}(S_{u_k})=\tau_{kQ'}^{-1}(S_{u_k})$, thus we have an
almost split sequence:
$$0\rightarrow U^1_{\ww'}\rightarrow E\rightarrow U_{\ww'}^{j-k+1}\rightarrow
0$$
Applying the reflection functor $R_{\ww''}^{-}:=R_{u_{k-1}}^{-}\circ\cdots\circ R_{u_1}^{-}
:\mod kQ'\to\mod kQ$ to this short exact sequence
we still get a short exact sequence:
$$0\rightarrow R_{\ww''}^{-}(U^1_{\ww'})\rightarrow R_{\ww''}^{-}(E)\rightarrow R_{\ww''}^{-}(U_{\ww'}^{j-k+1})\rightarrow
0$$  which is $$0\rightarrow U^k_\ww\rightarrow R_{\ww''}^{-}(E)\rightarrow U^j_\ww\rightarrow
0$$
and the left map is a left $\add\{R_{\ww''}^{-}(U^2_{\ww'}),\ldots,
R_{\ww''}^{-}(U^{j-k}_{\ww'})\}$-approximation, thus a left
$\add\{U^{k+1}_\ww,\ldots,U^{j-1}_\ww\}$-approximation.

 \end{proof}

\begin{rema}\label{remark to LequalT}
The statements $ L^j_\ww\simeq U^j_\ww\simeq T^j_\ww$ for $j=1,\ldots,l$ in
Theorem \ref{LequalT} is also true for non-reduced words $\ww=\ww's_{u_l}=c^{(0)}c^{(1)}\ldots c^{(m)}$ 
such that $\ww'$ is reduced and that all $c^{(t)}$ are subwords of $c$ whose supports satisfy 
$$supp(c^{(m)})\subseteq supp(c^{(m-1)})\subseteq \ldots \subseteq supp(c^{(1)})\subseteq supp(c^{(0)})\subseteq Q_0.$$
The proof above works without any change. Note that in this situation, the morphism $f^j_\ww$ is a monomorphism for $j=1,\ldots, l-1$, but the morphism $f^l_\ww$ may not be a monomorphism and $T^l_\ww$ and $U^l_\ww$ may be zero.
\end{rema}

\begin{cora}\label{propertiesofL}
Let $\ww$ be a $c$-sortable word, where $c$ is admissible with respect to the orientation of $Q$. Then the $kQ$-modules $L^j_\ww$ satisfy the following properties:
\begin{enumerate}
\item They are non-zero, indecomposable and pairwise non-isomorphic.
\item The space $\Hom_{kQ}(L^j_\ww,L^k_\ww)$ vanishes if $j>k$.
\end{enumerate}
\end{cora}

\begin{proof}
(1) Since $\ww$ is reduced, $L^j_\ww$ is non-zero by Proposition \ref{basicprop} (b).
Since reflection functors preserve isoclasses, the $U^j_\ww$ are indecomposable and pairwise non-isomorphic.

(2) Using reflection functors, we can assume that $U^k_\ww$ is simple
  projective, and then this is clear.

\end{proof}

\begin{thma}\label{Twtilting}
Let $(Q,c,\ww=s_{u_1}\ldots s_{u_l})$ be an admissible triple.
For $i\in Q_0^{(0)}$, denote by $t_\ww(i)$ the maximal integer such that $u_{t_\ww(i)}=i$. 
Let $$T_\ww:=\bigoplus_{i\in Q^{(0)}_0}L^{t_\ww(i)}_\ww\simeq \bigoplus_{i\in Q^{(0)}_0}T^{t_\ww(i)}_\ww.$$
\begin{enumerate}
\item $T_{\ww}$ is a tilting $kQ^{(0)}$-module.
\item We have $\Sub(T_\ww)=\add\{L^1_\ww,\ldots, L^{l}_\ww\}.$
\end{enumerate}
\end{thma}

\begin{proof}

(1) We prove that $T_\ww$ is a tilting $kQ^{(0)}$-module by induction on $l=l(\ww)$.

If $\ww=c^{(0)}$, then the assertion is clear since $T_\ww=kQ^{(0)}$ by definition.

We consider the case $\ww\neq c^{(0)}$.
Let $k$ be the maximal integer such that $k<l$ and $u_k=u_l$.
By the induction hypothesis we know that $T_{\ww'}=(T_\ww/T^l_\ww)\oplus T^k_\ww$ is a tilting $kQ^{(0)}$-module
where $\ww'$ is the word defined by $\ww=\ww's_{u_l}$.
By definition we have an exact sequence
$$\xymatrix{T^k_\ww\ar^f[r] & E\ar[r] & T^l_\ww\ar[r] & 0}$$
with a minimal left $\add\{T^{k+1}_\ww,\ldots,T^{l-1}_\ww\}$-approximation $f$.
Then $f$ is a minimal left $\add(T_\ww/T^l_\ww)$-approximation,
using the fact that $\ww$ is $c$-sortable and Corollary \ref{propertiesofL} (2).
Moreover $f$ is a monomorphism by Theorem \ref{LequalT}.
By Proposition \ref{tilting mutation}, we have that $T_\ww$ is a tilting $kQ^{(0)}$-module.

(2) We prove that $\Sub(T_\ww)=\add\{L^1_\ww,\ldots, L^{l}_\ww\}$ by induction on $l=l(w)$.

If $l(w)=1$, then the assertion is clear.

Assume that $l\geq 2$ and write $\ww=s_{u_1}\ww'$. 

\textit{Case 1: $u_1$ is in the support of $\ww'$}: this means that
$t_\ww(u_1)\geq 2$. Thus we have $$\begin{array}{rcl}T_\ww &=
  &\bigoplus_{i\in Q^{(0)}_0}U^{t_\ww(i)}_\ww\\
& = &\bigoplus_{i\in Q^{(0)}_0}R^{-}_{u_1}(U^{t_\ww(i)-1}_{\ww'})\\
& = &\bigoplus_{i\in Q^{(0)}_0}R^{-}_{u_1}(U^{t_{\ww'}(i)}_{\ww'})\\
& = & R^{-}_{u_1}(T_{\ww'})
\end{array}$$  
Then using the induction hypothesis we get 
$$\Sub T_{\ww'}=\add\{U^1_{\ww'},\ldots,U_{\ww'}^{l(\ww')}\}.$$
Moreover we have
$$\add\{U^2_{\ww},\ldots,U_{\ww}^{l(\ww)}\}=
R^{-}_{u_1}(\Sub T_{\ww'})\subset \Sub T_\ww\subset\add\{U^1_\ww,R^{-}_{u_1}(\Sub T_{\ww'})\}
=\add\{U^1_{\ww},\ldots,U_{\ww}^{l(\ww)}\}.$$

By definition of the $T^j_\ww$ there exists a short exact sequence:
$$\xymatrix{U^1_\ww=T^1_\ww\ar[r] & E\ar[r] & T^j_\ww\ar[r] & 0}$$ where $E$
is in $\add\{T^2_\ww,\ldots,T^{j-1}_\ww\}$ and where $j$ is the
minimal integer such that $u_j=u_1$ and $j>1$. It exists since $u_1$ is in
the support of $\ww'$.

The approximation map is a monomorphism by Theorem \ref{LequalT} (2), thus $U^1_\ww$ is in
$\Sub(E)\subset \Sub(T^2_\ww\oplus\ldots \oplus T^{j-1}_\ww)\subset \Sub T_\ww$.
Thus we have $\Sub T_\ww=\add\{U^1_{\ww},\ldots,U_{\ww}^{l(\ww)}\}$.

\textit{Case 2: $u_1$ is not in the support of $\ww'$.}

Then it is easy to see that 
$$T_\ww=U^1_\ww\oplus R^{-}_{u_1}(T_{\ww'}).$$
And we get  $$ \Sub T_\ww=\add\{U^1_\ww,R^{-}_{u_1}(U^1_{\ww'}),\ldots,R^{-}_{u_1}(U_{\ww'}^{l(\ww')})\}=\add\{U^1_\ww,U^2_\ww,\ldots ,U^{l(\ww)}_\ww\}.$$

\end{proof}

\begin{rema}
\begin{itemize}
\item[(a)]
The short exact sequence $\xymatrix{ L^k_\ww\ar@{>->}[r]^f & E\ar@{->>}[r] & L^j_\ww}$ in $\mod kQ$ is an almost split sequence in the category $\Sub(T_\ww)$.
\item[(b)] This almost split sequence is an element of $\Ext^1_\Lambda(L^j_\ww,L^k_\ww)$, which is the `2-Calabi-Yau complement' of the short exact sequence $\xymatrix{ L^j_\ww\ar@{>->}[r] & K\ar@{->>}[r] & L^k_\ww}$ of Proposition \ref{dimext1}.
\end{itemize}
\end{rema}

\begin{exa}\label{calculation of Sub T}
Let $Q$ be the quiver $\xymatrix@R=.1cm@C=.5cm{&2\ar[dr] &
  \\1\ar[rr]\ar[ur] &&3}$ and $\ww:=s_1s_2s_3s_1s_2s_1$
in Example \ref{calculation of T}.

The module $T_\ww$ is by definition $T^3_\ww\oplus T^5_\ww\oplus T^6_\ww$. It
is easy to check Theorem \ref{Twtilting}. The module $T_\ww$  is a tilting module over
$kQ$, and we have $$\Sub T_\ww=\{{\bsm 1\esm},\quad {\bsm 2\\1\esm},
\quad {\bsm
 &3&&\\1&&2&\\&&&1\esm},\quad  {\bsm
2&&3&&\\&1&&2&\\&&&&1\esm},\quad  {\bsm
\\&3&&&&&\\1&&2&&3&&\\&&&1&&2&\\&&&&&&1\esm},
\quad {\bsm
 3\\1\esm}\}.$$
\end{exa}

\subsection{Tilting modules with finite torsionfree class}

In this section we establish the converse of Theorem \ref{Twtilting}. Hence we get a natural bijection between tilting $kQ$-modules with finite torsionfree class and $c$-sortable elements in $W_Q$.

Let us start with some preparation.
To any (not necessarily reduced) word $\ww=s_{u_1}\ldots s_{u_l}=c^{(0)}c^{(1)}\ldots c^{(m)}$ such that all $c^{(t)}$ are subwords of $c$ whose supports satisfy 
$$supp(c^{(m)})\subseteq supp(c^{(m-1)})\subseteq \ldots \subseteq supp(c^{(1)})\subseteq supp(c^{(0)})\subseteq Q_0,$$
we can associate $kQ^{(0)}$-modules $T^j_\ww$ for $j=1,\ldots,l$ and $T_\ww$ in the same way as in the $c$-sortable case.

\begin{lema}\label{csortableifftilting}
Let $\ww=s_{u_1}\ldots s_{u_l}=\ww's_{u_l}$ be as above. Assume that $\ww$ is non-reduced and that $\ww'$ is reduced.
Then the number of indecomposable summands of $T_\ww$ is strictly less than $l(c^{(0)})$. 
\end{lema}

\begin{proof}
By Remark \ref{remark to LequalT}, we have $T^l_{\ww}\simeq L^l_{\ww}$.
Since $\ww$ is not reduced, this is zero by Proposition \ref{basicprop} (b). Since $\ww'$ is reduced, all $T^j_\ww\simeq L^j_\ww$ ($j\neq l$) are indecomposable by Theorem \ref{2spherical}. Therefore we have the assertion.
\end{proof}

\begin{lema}\label{AR things}
Let $Q$ be an acyclic quiver and $T$ be a tilting $kQ$-module. 
\begin{enumerate}
\item The category $\Sub{T}$ has almost split sequences.
\item If $\Sub{T}$ has finitely many indecomposable modules, then the AR-quiver of $\Sub{T}$ is
a full subquiver of the translation quiver $\mathbb{Z} Q$.
\end{enumerate}
\end{lema}

\begin{proof}
(1) This is well-known \cite{AS81}.

(2) We can clearly assume that $Q$ is connected. Since $T$ is a tilting
module, then all indecomposable projectives are in $\Sub T$. The
irreducible maps between projectives in $\Sub T$ coincide with the
irreducible maps between projectives in $\mod kQ$, so that $Q$ is a full
subquiver of the AR-quiver of $\Sub T$.  Moreover, for any indecomposable module in $\Sub T$, there is a nonzero map from an indecomposable projective module. Since $\Sub T$ is of finite type and $Q$ is connected, it follows that the AR-quiver of $\Sub T$ is connected.

We now claim that each indecomposable module in $\Sub T$ is of the form
$\tau^{-t}P$, where $\tau$ is the AR-translate in $\Sub T$ and $P$ is
indecomposable projective. If not, then since $\Sub T$ is of finite type,
there is some $\tau$-periodic indecomposable $X$.
Then, since the quiver of $\Sub T$ is connected, there must be an irreducible map between some periodic
indecomposable $X$ and some $\tau^{-t}P$ with $P$ indecomposable
projective. Applying $\tau^t$ we can assume that the second
module is $P$. If $f:X\to P$ is irreducible, then $X$ is projective, a
contradiction. If $g:P\to X$ is irreducible, then $h:\tau X\to P$ is
irreducible, so $\tau X$ is projective, a contradiction. Thus each indecomposable of $\Sub T$ is of the form $\tau^{-t}P$, where $P$ is indecomposable projective. 

Then using the fact that $Q$ is a full subquiver of the AR-quiver of $\Sub T$, we deduce that the AR-quiver of $\Sub T$ is a full subquiver of $\mathbb{Z}Q$.  

\end{proof}

From Lemmas \ref{csortableifftilting} and \ref{AR things} we deduce a nice consequence.

\begin{thma}\label{SubTfinite}
Let $Q$ be an acyclic quiver. Let $c$ be a Coxeter element admissible with respect to the orientation of $Q$. Let
$T$ be a tilting $kQ$-module. Assume that $\Sub{T}$ has finitely
many indecomposable modules. Then there exists a unique $c$-sortable
word $\ww$ such that $T_\ww\simeq T$.
\end{thma}

\begin{proof}

Without loss of generality, we assume $c=s_1s_2\ldots s_n$.
We denote by $\tau$ the AR-translation of $\Sub T$.
For any $i\in Q_0$, we denote by $m(i)$ the minimal number satisfying $\tau^{-m(i)-1}(e_i kQ)=0$, which exists by Lemma \ref{AR things}.
Then for $t\geq 0$ we look at the set 
$$\{i\in Q_0\ |\ \tau^{-t}(e_i kQ)\neq 0\}=\{i_1^{(t)} < i_2^{(t)}<\cdots <i_{p_t}^{(t)}\}$$ 
and define $c^{(t)}:=s_{i_1^{(t)}}s_{i_2^{(t)}}\ldots s_{i_{p_t}^{(t)}}.$ Then the word $\ww:= c^{(0)}c^{(1)}\ldots c^{(m)}$ where $m:=\max \{m(i)\ |\ i\in Q_0\}$ satisfies
$$supp(c^{(m)})\subseteq \ldots \subseteq supp(c^{(1)})\subseteq supp(c^{(0)}).$$

For each expression $\ww=\ww'\ww''$, we define $m_{\ww'}(i)+1$ as the number of $s_i$ ($i\in Q_0$) appearing in $\ww'$.
By using induction on $l(\ww')$, we have
$$T_{\ww'}\simeq \bigoplus_{i\in Q_0}\tau^{-m_{\ww'}(i)}(e_i kQ)$$
by using the almost split sequences in $\Sub T$ and the shape of the AR-quiver of $\Sub T$ given in Lemma \ref{AR things} (2).
In particular the number of indecomposable direct summands of $T_{\ww'}$ is exactly $n$
since $m_{\ww'}(i)\le m(i)$ for any $i\in Q_0$.
Moreover we have $T_\ww\simeq\bigoplus_{i\in Q_0}\tau^{-m(i)}(e_i kQ)\simeq T$
since $m_\ww(i)=m(i)$ for any $i\in Q_0$.

We only have to check that $\ww$ is reduced. 
Otherwise we take an expression $\ww=\ww'\ww''$ such that $\ww'$ is non-reduced and $l(\ww')$ is minimal with this property.
By Lemma \ref{csortableifftilting}, the number of indecomposable direct summands of
$T_{\ww'}$ is less than $n$, a contradiction. Thus $\ww$ is reduced.
\end{proof}

 As a consequence we get the following: 
\begin{cora} If $T$ is a tilting $kQ$-module such that $\Sub T$ is of finite type, then  all indecomposables in $\Sub T$ are rigid as $kQ$-modules.
\end{cora}

Combining Theorem \ref{SubTfinite} with Theorem \ref{Twtilting} we get the following result which was first proved using other methods in \cite{Tho}.

\begin{cora} There is   1-1
correspondences 
\begin{itemize}
\item[(a)]$\xymatrix{ \{\textrm{finite torsionfree classes of
  }\mod kQ\textrm{ containing $kQ$}\}\ar^-{1:1}@{<->}[r] & \{ c \textrm{-sortable words with }c^{(0)}=c\}}.$
  \item[(b)]$\xymatrix{ \{\textrm{finite torsionfree classes of }\mod kQ\}\ar^-{1:1}@{<->}[r] & \{ c \textrm{-sortable words}\}}.$
  \end{itemize}
\end{cora}

\subsection{Co-$c$-sortable situation}
Dually, we can state the defintion.
\begin{dfa}
 Let $c$ be a Coxeter element of the Coxeter group $W_Q$ admissible with respect to the orientation of $Q$. 
An element $w$ of $W_Q$ is called \emph{co-c-sortable} if there exists a reduced expression $\ww=s_{u_l}\ldots s_{u_1}$ of $w$ of the form $\ww=c^{(m)}c^{(m-1)}\ldots c^{(0)}$ where all $c^{(t)}$ are subwords of $c$ whose supports satisfy 
$$supp(c^{(m)})\subseteq supp(c^{(m-1)})\subseteq \ldots \subseteq supp(c^{(1)})\subseteq supp(c^{(0)})\subseteq Q_0.$$
Note that this definition is equivalent to the fact that $\ww^{-1}=s_{u_1}\ldots s_{u_l}$ is $c^{-1}$-sortable.

We denote by $Q^{(j)}$ the quiver $Q$ restricted to the support of $c^{(j)}$.
\end{dfa}

From a $c$-sortable word $\ww$ we define in this subsection, $kQ^{(0)}$-modules $T^\ww_j$, $T^\ww$ and $U^\ww_j$  in a dual manner to the modules $T^j_\ww$,  $T_\ww$ and $U^j_\ww$ defined in the previous subsections.
 
 \medskip
 
For $j=1,\ldots, l$, we define $kQ^{(0)}$-modules $T^\ww_j$. For $1\leq
j\leq l(c^{(0)})$, $T_j^\ww$ is the injective $kQ^{(0)}$-module $e_{u_j}D(kQ^{(0)})$. For
$j>l(c^{(0)})$, let $k$ be the maximal integer such that $k<j$ and
$u_k=u_j$. We define $T_j^\ww$ as the kernel of the map
$$ f_j^\ww:E\rightarrow T_k^\ww$$
where $f_j^\ww$ is a minimal right $\add\{T_{k+1}^\ww,\ldots,T^\ww_{j-1}\}$-approximation.

\medskip

Then we define a $kQ$-module $T^\ww$ as the direct sum $T^\ww=\bigoplus_{i\in Q^{(0)}_0}T^\ww_{t_\ww(i)}$, where $t_\ww(i)$ is the maximal integer such that $u_{t_\ww(i)}=i$.

\medskip
A \emph{co-admissible triple} is a triple $(Q,c,\ww)$ consisting of an
acyclic quiver $Q$, a Coxeter element $c$ admissible with the
orientation of $Q$, and a co-$c$-sortable word $\ww=c^{(m)}c^{(m-1)}\ldots c^{(0)}$.

Let $(Q,c,\ww)$ be a co-admissible triple with $\ww=s_{u_l}s_{u_{l-1}}\ldots
s_{u_1}$. For $j=1,\ldots, l$, we define $kQ$-modules $U^\ww_j$ by induction on $l$. 

If $l=1$ then we define $U_1^\ww=e_{u_1}D(kQ^{(0)})$, the simple injective $kQ^{(0)}$-module associated to the vertex $u_1$. 

Assume $l\geq 2$. Then we write $\ww=\ww' s_{u_1}$.
It is not hard to check that the triple
\[(Q'=\mu_{u_1}(Q^{(0)}), s_{u_1}c^{(0)}s_{u_1}, \ww')\]
is a co-admissible triple with $l(w')=l-1$. Therefore by the induction hypothesis we have $kQ'$-modules $U_1^{\ww'},\ldots U_{l-1}^{\ww'}$. For $j=2,\ldots, l$ we define 
$$U_j^\ww =R_{u_1}(U_{j-1}^{\ww'})$$
where $R_{u_1}$ is the reflection functor
$$ \xymatrix{\mod kQ'=\mod k(\mu_{u_1}Q^{(0)})\ar[rr]^(.6){R_{u_1}} && \mod kQ^{(0)}}$$ 
at the sink $u_1$ of $Q^{(0)}$.

Then a dual version of Theorems \ref{LequalT} (2) and \ref{Twtilting} hold. More precisely we have the following.

\begin{thma}
Let $Q$ be an acyclic quiver and $c$ be the associated Coxeter element. Then for a co-$c$-sortable element $\ww$ the following hold.
\begin{itemize}
\item[(a)] For all $j=1, \ldots, l$, we have $T^\ww_j\simeq U^\ww_j$, and $f^\ww_j$ is an epimorphism.
\item[(b)] The subcategory $\Fac(T^\ww)$ is finite, and $\Fac(T^\ww)=\add\{T^\ww_1,\ldots,T^\ww_l\}$.
\end{itemize}
\end{thma}
Dually to Theorem \ref{SubTfinite}, we also have the following.
\begin{thma}
For any tilting module $T\in\mod kQ$ such that $\Fac T$ is finite, then the following hold.
\begin{itemize}
\item[(a)] The AR-quiver of $\Fac (T)$ is a full subquiver of $\mathbb{Z}Q$.
 \item[(b)] There exists a (unique) co-$c$-sortable word $\ww$ (which can be constructed from the AR-quiver of $\Fac (T)$) such that $T\simeq T^\ww$.
 \end{itemize}
\end{thma}
\begin{cora} There is  1-1
correspondences 
\begin{itemize}
\item[(a)]$\xymatrix{ \{\textrm{finite torsion classes of
  }\mod kQ\textrm{ containing $D(kQ)$}\}\ar^-{1:1}@{<->}[r] & \{ \textrm{co-}c \textrm{-sortable words with }c^{(0)}=c\}}.$
  \item[(b)]$\xymatrix{ \{\textrm{finite torsion classes of
  }\mod kQ\}\ar^-{1:1}@{<->}[r] & \{ \textrm{co-}c \textrm{-sortable words}\}}.$
  \end{itemize}
\end{cora}

Let $\ww$ be a co-$c$-sortable word with $c^{(0)}=c$. We denote by $A_\ww$ (resp. $\overline{A}_\ww$) the Auslander algebra of the category $\Fac (T^\ww)$ (resp. $\overline{\Fac}(T^\ww)=\Fac (T^\ww)/ \add(DkQ)$), that is $$A_\ww=\End_{kQ}(\bigoplus_{j=1}^l T^\ww_j) \quad \textrm{and} \quad \overline{A}= \overline{\End}_{kQ}(\bigoplus_{j=1}^l T^\ww_j)=\End_{kQ} (\bigoplus_{j=l(c)+1}^l T^\ww_j).$$  
Following \cite[Theorem 5.21]{Ami3} we get the following result:
\begin{thma}\label{equivalence}
There exists a commutative diagram of triangle functors:
$$\xymatrix@C=2cm{\Dd^b(\overline{A}_\ww)\ar[r]^-{\textrm{Res.}} \ar[d]& \Dd^b(A_\ww)\ar[r]^{-\lten_{A_\ww}M_\ww} & \Dd^b(\Lambda_w)\ar[d] \\
\Cc_2(\overline{A}_\ww)\ar[rr]^F && \underline{\Sub}\Lambda_w
 ,}$$ where $\Cc_2(\overline{A}_\ww)$ is the generalized $2$-cluster category defined in \cite{Ami3} associated with the algebra $\overline{A}_\ww$ of global dimension at most 2, and where $F$ is an equivalence of categories.  
\end{thma}
The proof in \cite{Ami3} deals with $T^\ww$ in the preinjective component of the AR-quiver of $\mod kQ$ (that is $\End_{kQ}(T^\ww)$ concealed), but the proof only uses the fact that $\Fac(T^\ww)$ is finite and that the AR-quiver of $\Fac(T^\ww)$ is a full subquiver of $\mathbb{Z}Q$.   

\medskip
Note that for any element $w$, the $2$-CY category $\underline{\Sub}\Lambda_w$ is equivalent to a generalized $2$-cluster category $\Cc_2(A)$ \cite{ART}, but the algebra $A$ of global dimension 2 is constructed in a very different way. A link between the construction given here and the construction of \cite{ART} is given in \cite{Ami4}. 

\begin{rema}
The proof in \cite{Ami3} does not carry over for $c$-sortable words. Indeed, for a general co-$c$-sortable word, the AR-quiver of $\Fac (T^\ww)$ is a subquiver of the quiver of $\End_\Lambda (M_\ww)$, fact which is used in the proof. For a $c$-sortable word $\ww$, the AR-quiver of $\Sub (T_\ww)$ is not a subquiver of the quiver of $\End_\Lambda (M_\ww)$.
\end{rema}

\section{Problems and examples}

In this section we discuss some possible generalizations of the description of the layers in terms of tilting modules, beyond the $c$-sortable case. We pose some problems and give some examples to illustrate limitations for what might be true.

Recall from Section 2 that to a reduced expression $\ww$ of an element $w$ in $W_Q$ we have associated a set $\{L^j_\ww\}$ of $l(w)$ indecomposable rigid $\Lambda$-modules which we call layers, and which are indecomposable rigid $kQ$-modules when $\ww$ is $c$-sortable, where $c$ is admissible with respect to the orientation of $Q$. Under the same assumption (\emph{i.e.} $\ww$ is $c$-sortable), we constructed a set $\{T^j_\ww\}$ of $l(w)$ indecomposable $kQ$-modules via minimal left approximations, starting with the tilting module $kQ$, and ending up with a tilting module $T_\ww$. All minimal left approximations were monomorphisms. We showed that the two sets of indecomposable modules coincide. In particular, the module $L_\ww:=L^{t_\ww(1)}_\ww\oplus \cdots \oplus L^{t_\ww(n)}_\ww$, where for $i\in Q_0=\{1,\ldots,n\}$ the integer $t_\ww(i)$ is the position of the last reflection $s_i$ in the word $\ww$, is a tilting module over $kQ$.
 
We now consider the case of words $\ww$ with the assumption that $\ww=c\ww'$, where $c$ is a Coxeter element admissible with respect to the orientation of $Q$. When $\ww=c s_{u_{n+1}}\ldots s_{u_l}$ is a word, we define $T_\ww$ to be a tilting module associated with $\ww$ if it is possible to carry out the following. Start with $kQ=P_1\oplus \cdots \oplus P_n$, where $P_i$ is the indecomposable projective $kQ$-module associated with the vertex $i$. If possible, exchange $P_{u_{n+1}}$ with a non-isomorphic indecomposable $kQ$-module to get a tilting module $T'=kQ/P_{u_{n+1}}\oplus P_{u_{n+1}}^*$, then replace summand number $i_2$ in $T'$ by a non-isomorphic indecomposable $kQ$-module to get a new tilting module $T''$, etc. If an exchange is possible at each step, we obtain a tilting module $T_\ww$. We say that a word $\ww=c\ww'$ starting with a Coxeter element is \emph{tilting} if $T_\ww$ exists, and $\ww$ is \emph{monotilting} if morever $T_\ww$ is obtained by only using left
  approximations. Hence $c$-sortable words are examples of monotilting words.

It is natural to ask the following question about tilting and monotilting words.

\begin{pb}:\begin{itemize}
\item[(a)] Characterize the tilting words $\ww$. In particular is every reduced word $\ww=c\ww'$ starting with a Coxeter element tilting?
\item[(b)] Characterize the monotilting words.
\item[(c)] When do two tilting words $\ww_1$ and $\ww_2$ give rise to the same tilting module? Or formulated differently, for which tilting words $\ww$ do we have $T_\ww\simeq kQ$? 
                    \end{itemize}
\end{pb}

Note that all these questions can also be translated into combinatorial problems for acyclic cluster algebras.

Note that non-reduced words may be monotilting as the following example shows.

\begin{exa}
 Let $Q$ be the quiver $\xymatrix@-.5cm@R=-1mm{ &2\ar[dr] & \\ 1\ar[dr]\ar[ur] && 4\\ & 3\ar[ur]&} $ and $\ww:=s_1s_2s_3s_4s_3s_1s_4$. Then $\ww$ is not reduced. One can easily check that we have 
 $$T^1_{\ww}={\bsm 1\esm}, \ T^2_{\ww}={\bsm 2\\1\esm}, \ T^3_{\ww}={\bsm 3\\1\esm}, \ T^4_{\ww}={\bsm &&4&&\\ &2&&3&\\1&&&&1\esm}, \ T^5_{\ww}={\bsm 4\\2\\1\esm}.$$
 Then the minimal left $\add(T^2_\ww\oplus T^4_\ww \oplus T^5_\ww )$-approximation of $T^1_{\ww}$ is ${\bsm 1\esm}\to {\bsm 2\\1\esm}\oplus{\bsm &&4&&\\ &2&&3&\\1&&&&1\esm}$. It is a monomorphism whose cokernel is $T^6_{\ww}={\bsm &&4&&&\\ &2&&3&&2\\1&&&&1&\esm}.$ The minimal left $\add(T^2_\ww\oplus T^5_\ww \oplus T^6_\ww )$-approximation of $T^4_{\ww}$ is ${\bsm &&4&&\\ &2&&3&\\1&&&&1\esm}\to{\bsm &&4&&&\\ &2&&3&&2\\1&&&&1&\esm}$. It is a monomorphism whose cokernel is $T^7_\ww={\bsm 2\esm}$.
 Hence $\ww$ is monotilting. 
\end{exa}

Recall that in the $c$-sortable case, then $\ww$ is monotilting and $\Sub T_\ww$ is of finite type. This is not the case in general.

\begin{exa}
Let $Q$ be the quiver $\xymatrix@-.5cm@R=-1mm{ &2\ar[dr] & \\ 1\ar[dr]\ar[ur] && 4\\ & 3\ar[ur]&} $ and $\ww:=s_1s_2s_3s_4s_2s_3s_4s_1$. Then one can show that $\ww$ is monotilting and that 
$T_\ww= {\bsm 4&&&&4\\ &3&&2&\\ &&1&&\esm}\oplus {\bsm 4\\3\\1\esm}\oplus {\bsm 4\\2\\1\esm}\oplus {\bsm 4\esm}.$
Then one can check easily that all the modules of the form ${\bsm &&4&&\\&2&&3&\\ 1&&&&1\esm},\ {\bsm &&4&&&&4&&\\&2&&3&&2&&3&\\1&&&&1&&&&1\esm},\ {\bsm &&4&&&&4&&&&4&&\\&2&&3&&2&&3&&2&&3&\\1&&&&1&&&&1&&&&1\esm},\ldots$ are in $\Sub({\bsm 4&&&&4\\ &3&&2&\\ &&1&&\esm})$.
\end{exa}

However, it may happen that $\Sub T_\ww$ is of finite type for a tilting word $\ww$ which is not $c$-sortable. It follows from Theorem \ref{SubTfinite} that there exists a unique $c$-sortable word $\tilde{\ww}$ such that $T_\ww=T_{\tilde{\ww}}$. We then pose the following.

\begin{pb}:
\begin{itemize}
 \item[(a)] Characterize the tilting words $\ww$ with $\Sub T_\ww$ finite.
\item[(b)] For such words $\ww$, how can we construct the unique $\tilde{\ww}$ such that $T_\ww=T_{\tilde{\ww}}$?
\end{itemize}
\end{pb}

When $\ww$ is monotilting, we have $$\{T^1_\ww,\ldots, T^{l(\ww)}_\ww\}\subseteq \Sub T_\ww=\Sub T_{\tilde{\ww}}=\add \{T^1_{\tilde{\ww}}, \ldots, T^{l(\tilde{\ww})}_{\tilde{\ww}}\}.$$ Hence $l(\ww)\leq l(\tilde{\ww})$ and we expect that $\tilde{\ww}$ is obtained by enlarging some rearrangement of $\ww$.

\begin{exa}\label{examplecompletion}
Let $Q$ be the quiver $\xymatrix@-.5cm@R=-1mm{ &2\ar[dr] & \\ 1\ar[dr]\ar[ur] && 4\\ & 3\ar[ur]&} $ and $\ww:=s_1s_2s_3s_4s_2s_3s_1s_4$. Then $\ww$ is monotilting and we have $$T_\ww= {\bsm &&4&&&&4&&\\&2&&3&&2&&3&\\1&&&&1&&&&1\esm} \oplus {\bsm 4\\3\\1\esm}\oplus {\bsm 4\\2\\1\esm}\oplus {\bsm &&4&&&&4&&&&4&&\\&2&&3&&2&&3&&2&&3&\\1&&&&1&&&&1&&&&1\esm}.$$
Then $\ww$ is not $c$-sortable, $\Sub T_\ww$ is finite and one can check that $\tilde{\ww}=s_1s_2s_3s_4s_1s_2s_3s_4s_1s_2s_3s_4s_2s_3$.
\end{exa}

When $\ww$ is $c$-sortable, $\ww$ is a monotilting word and $T_\ww$ coincides with $L_\ww$ given by the layers. In general $L_\ww$ is not a $kQ$-module, but as we have seen there is an indecomposable $kQ$-module associated with each indecomposable summand of $L_\ww$, and hence a $kQ$-module $(L_\ww)_Q$ associated with $L_\ww$. In this connection we have the following questions:

\begin{pb}:
\begin{itemize}\item[(1)] For which $\ww$ does the following hold ?
\begin{itemize}\item[(a)] each indecomposable summand of $(L_\ww)_Q$ is rigid,
\item[(b)] $(L_\ww)_Q$ is a tilting module,
\item[(c)] $\ww$ is tilting and $T_\ww=(L_\ww)_Q$,
\item[(d)] $\ww$ is monotilting and $T_\ww=(L_\ww)_Q$.
\end{itemize}
\item[(2)] If $\ww$ is monotilting and $(L_\ww)_Q$ is rigid, do we have $T_\ww=(L_\ww)_Q$?
\end{itemize}
\end{pb}

As we already saw in Example \ref{exampleLQnonrigid}, it can happen that (a) fails. In this example, one can check that $\ww$ is monotilting.

\begin{exa}
 Let $Q$ be the quiver $\xymatrix@-.5cm@R=-1mm{ &2\ar[dr] & \\ 1\ar[ur]\ar[rr] && 3}$, and $\ww:=s_1s_2s_3s_2s_1s_2$. The word $\ww'=s_1s_2s_3s_2s_1$ is monotilting and we have
$T_{\ww'}={\bsm &3&&&&&\\1&&2&&3&&\\&&&1&&2&\\&&&&&&1\esm}\oplus {\bsm 3\\1\esm}\oplus {\bsm &3&&\\1&&2&\\&&&1\esm}$
To exchange ${\bsm 3\\1\esm}$ we have to use the minimal right approximation $\xymatrix{g:{\bsm &3&&&&&\\1&&2&&3&&\\&&&1&&2&\\&&&&&&1\esm}\ar[r] & {\bsm 3\\1\esm}}.$ Hence $\ww$ is a tilting word which is not monotilting and we get
$T_\ww={\bsm &3&&&&&\\1&&2&&3&&\\&&&1&&2&\\&&&&&&1\esm}\oplus {\bsm 2&&3&&\\&1&&2&\\&&&&1\esm}\oplus {\bsm &3&&\\1&&2&\\&&&1\esm}.$
The cluster-tilting object $M_\ww$ of $\Sub \Lambda_w$ associated with $\ww$ has the indecomposable summands:
$$M_\ww:={\bsm 1\esm}\oplus {\bsm 2\\1\esm}\oplus {\bsm &3&&\\1&&2&\\&&&1\esm}\oplus {\bsm &&2&\\&3&&1\\1&&&\esm}\oplus {\bsm &&&1&&&\\&&2&&3&&\\&3&&1&&2&\\1&&&&&&1\esm}\oplus {\bsm &&&&2&&\\&&&1&&3&\\&&3&&2&&1\\&2&&1&&&\\1&&&&&&\esm}$$
We then see that $T_\ww=(L_\ww)_Q$, even though $\ww$ is not a monotilting word.
\end{exa}

\begin{exa} Let $Q$ and $\ww$ be as in Example \ref{examplecompletion}. Then we have $$M_\ww={\bsm 1\esm}\oplus{\bsm 2\\1\esm}\oplus {\bsm 3\\1\esm}\oplus{\bsm &&4&&\\&2&&3&\\1&&&&1\esm}\oplus {\bsm &2&&&\\1&&4&&\\&&&3&\\&&&&1\esm}\oplus {\bsm &3&&&\\1&&4&&\\&&&2&\\&&&&1\esm}\oplus {\bsm &&&&1&&&&\\&&&2&&3&&&\\&&4&&1&&4&&\\&3&&&&&&2&\\1&&&&&&&&1\esm}\oplus {\bsm &&4&&\\&2&&3&\\1&&4&&1\esm}.$$
Therefore we obtain
$(L_\ww)_Q={\bsm &&4&&&&4&&\\&3&&2&&3&&2& \\ 1&&&&1&&&&1\esm}\oplus{\bsm 4\\3\\1\esm}\oplus {\bsm 4\\2\\1\esm}\oplus {\bsm 4 \esm}.$
Each indecomposable summand is rigid, but $(L_\ww)_Q$ is not a tilting module. Therefore we can have (a) without (b).
 \end{exa}



\newcommand{\etalchar}[1]{$^{#1}$}
\def\cprime{$'$}
\providecommand{\bysame}{\leavevmode\hbox to3em{\hrulefill}\thinspace}
\providecommand{\MR}{\relax\ifhmode\unskip\space\fi MR }
\providecommand{\MRhref}[2]{%
  \href{http://www.ams.org/mathscinet-getitem?mr=#1}{#2}
}
\providecommand{\href}[2]{#2}

\end{document}